\documentclass[11pt,final]{amsart}

\usepackage{amssymb,amsmath,latexsym}
\usepackage{amsthm}
\usepackage{amsfonts}
\usepackage{enumerate}
\usepackage{booktabs}
\usepackage{mathrsfs}
\usepackage{comment}

\usepackage{manfnt}
\usepackage{hyperref}
\usepackage{xcolor}
\usepackage{color}
\hypersetup{
    colorlinks, linkcolor={red!80!black},
    citecolor={blue!80!black}, urlcolor={blue}
}
\usepackage[color]{showkeys}
\definecolor{refkey}{gray}{.75}
\definecolor{labelkey}{gray}{.75}

\usepackage{graphicx,pgf}
\usepackage{float}

\usepackage{caption}
\usepackage{subcaption}

\usepackage{fancyhdr}

\newtheorem{theorem}{Theorem}[section]

\newtheorem{definition}{Definition}[section]
\newtheorem{proposition}{Proposition}[section]
\newtheorem{lemma}{Lemma}[section]
\newtheorem{remark}{Remark}[section]

\newcommand{\aiminabs}[1]{\lvert #1 \rvert}
\newcommand{\aiminnorm}[1]{\| #1 \|}
\newcommand{\aimininner}[2]{\langle #1, #2 \rangle}

\numberwithin{equation}{section}
\numberwithin{figure}{section}

\let\oldtocsection=\tocsection
\let\oldtocsubsection=\tocsubsection
\let\oldtocsubsubsection=\tocsubsubsection
\renewcommand{\tocsection}[2]{\hspace{0em}\oldtocsection{#1}{#2}}
\renewcommand{\tocsubsection}[2]{\hspace{2em}\oldtocsubsection{#1}{#2}}
\renewcommand{\tocsubsubsection}[2]{\hspace{4em}\oldtocsubsubsection{#1}{#2}}

\begin{document}
\title{The finite dimensions and determining modes of the global attractor for 2d Boussinesq equations with fractional Laplacian}
\author{Aimin Huang}
\author{Wenru Huo}
\address{The Institute for Scientific Computing and Applied Mathematics, Indiana University, 831 East Third Street, Rawles Hall, Bloomington, Indiana 47405, U.S.A.}
\address{The Department of Mathematics, Indiana University, 831 East Third Street, Rawles Hall, Bloomington, Indiana 47405, U.S.A.}
\email{AH:huangepn@gmail.com}
\email{WH:whuo@imail.iu.edu}


\date{\today}
\begin{abstract}
In this article, we prove the finite dimensionality of the global attractor and estimate the numbers of the determining modes for the 2D Boussinesq system in a periodic channel with fractional Laplacian in subcritical case.
\end{abstract}

\maketitle

\setcounter{tocdepth}{2}
\tableofcontents
\addtocontents{toc}{~\hfill\textbf{Page}\par}
\section{Introduction}
This paper estimates the number of determining modes and the dimension of the global attractor for the two-dimensional (2D) incompressible Boussinesq equations with subcritical dissipation. The 2D Boussinesq equations read
\begin{equation}\begin{cases}\label{eq1.1.1}
\partial_t\boldsymbol u + \boldsymbol u\cdot \nabla \boldsymbol u + \nu(-\Delta)^\alpha \boldsymbol u=- \nabla \pi + \theta \boldsymbol e_2, \qquad x \in \Omega, \hspace{2pt} t > 0, \\
\nabla \cdot \boldsymbol u=0,  \qquad\qquad\qquad\qquad\qquad\qquad\qquad x \in \Omega, \hspace{2pt} t > 0,\\
\partial_t \theta +\boldsymbol u\cdot \nabla \theta+ \kappa(-\Delta)^{\beta}\theta = f, \qquad\qquad\qquad x \in \Omega, \hspace{2pt} t > 0,
\end{cases}\end{equation}
where $\Omega=[0, 2\pi]^2$ is the periodic domain, $\nu>0$ the fluid viscosity, and $\kappa>0$ the  diffusivity;  $\boldsymbol u =\boldsymbol u(x,t)=(u_1(x,t), u_2(x,t))$ denotes the velocity, $\pi=\pi(x,t)$  the pressure, $\theta= \theta(x,t)$ a scalar function  which may for instance represents the temperature variation in the content of thermal convection, $\boldsymbol e_2=(0,1)$  the unit vector in the vertical direction, and $f=f(x)$ a time-independent forcing term. We associate to \eqref{eq1.1.1} the following initial data
\begin{equation} \label{eq1.1.2}
\boldsymbol u(x,0)=\boldsymbol u_0(x),\quad \theta(x,0)=\theta_0(x), \qquad \qquad x \in \Omega.
\end{equation}
Since in this article we consider 2D Boussinesq equations with a subcritical dissipation, we assume that the exponents $\alpha$ and $\beta$ satisfy
\begin{equation} \label{e1.2}
\alpha,\; \beta \in (\frac{1}{2},1).
\end{equation}
Additionally, along with \cite{HH15}, we also assume that
\begin{equation}\label{cond1}
s_1> 2\mathrm{max}\{1-\alpha, 1-\beta \}, \qquad s_2 \geq 1,
\end{equation}
and
\begin{equation} \label{cond2}
0 \leq s_2-s_1 < \alpha+\beta.
\end{equation} 
Moreover, integrating \eqref{eq1.1.1} on $\Omega$ and integration by parts yield
\[
\frac{\mathrm{d}}{\mathrm{d}t} \bar{\boldsymbol u} = 
\frac{1}{|\Omega|} \frac{\mathrm{d}}{\mathrm{d}t}  \int_{\Omega}
\boldsymbol u \mathrm{d}x = \bar{\theta} \boldsymbol e_2, \qquad
\frac{\mathrm{d}}{\mathrm{d}t} \bar{\theta} = 
\frac{1}{|\Omega|} \frac{\mathrm{d}}{\mathrm{d}t}  \int_{\Omega}
\theta \mathrm{d}x = \bar{f},
\]
where $\bar{\boldsymbol u}, \, \bar{\theta}, \, \bar{f}$ are the mean of $\boldsymbol u, \, \theta, \, f$ over $\Omega$ respectively; that is 
\[
\bar{\boldsymbol u} \equiv \frac{1}{|\Omega|} \int_{\Omega} \boldsymbol u \mathrm{d}x, \qquad \bar{\theta} \equiv \frac{1}{|\Omega|} \int_{\Omega} \theta \mathrm{d}x, \qquad
\bar{f} \equiv \frac{1}{|\Omega|}  \int_{\Omega} f \mathrm{d}x.
\]
Therefore, with loss of generality, we assume that $\boldsymbol u, \, \theta, \, f$ are all of mean zero. Otherwise, we can replace $\boldsymbol u-\bar{\boldsymbol u}, \, \theta-\bar{\theta}, \, f-\bar{f}$ by $\boldsymbol u, \, \theta, \, f$ respectively.

Recently, the 2D Boussinesq equations and their fractional generalizations have attracted considerable attention due to their physical applications and mathematical chanlleges. When $\alpha=\beta=1$, the system \eqref{eq1.1.1} is then called the standard 2D Boussinesq equations, which  are widely used to model the geophysical flows such as atmospheric fronts and oceanic circulation and also play an important role in the study of Rayleigh-B\'enard convection (c.f. \cite{Ped87}). Flows which travel upwards in the middle atmosphere change because of the changes of atmospheric properties. This anomalous phenomenon can be modeled by using the fractional Laplacian. Moreover, some models with fractional Laplacian such as the surface quasi-geostrophic equations and Boussinesq equation have very significant applications. In the mathematical respect, the global well-posedness, global regularity of the standard 2D Boussinesq system as well as the existence of the global attractor have been widely studied, see for example \cite{FMT87, Wang05, Wang07, YJW14}.

This work is motivated by the \cite{JT15}, where the finite dimensionality of the global attractor for 3D primitive equations has been proved, and it is a natural continuation of \cite{HH15}, where we proved the existence of global attractor of the 2D Boussinesq equations. The aim of this article is twofold. We first prove the finite dimensionality of the global attractor of system \eqref{eq1.1.1} by showing that the strong solutions of \eqref{eq1.1.1} on the global attractor satisfying the 
Ladyzhenskaya squeezing property. The second goal is to improve the estimates for the number of determining modes of the global attractor for the system \eqref{eq1.1.1}. Moreover, we prove that there is a finite number $m$, such that each trajectory $(\theta(t), \boldsymbol u(t))$ of strong solutions on the global attractor is uniquely determined by the its projection $P_m(\theta(t), \boldsymbol u(t))$ onto the space generated by $\{\omega_1, \cdots, \omega_m \}$, which are the first $m$ eigenfunctions of the operator $\Lambda$. 

The roadmap of this article is as follows. In Section \ref{sec2}, we introduce the notations, some preliminary results, state our main results, as well as the results from \cite{HH15} about the existence of the global attractor in the certain Sobolev space.  Section \ref{sec3} is devoted to proving that the global attractor $\mathcal{A}$ has finite Hausdorff and fractal dimensions. In Section \ref{sec4}, we prove the existence of the absorbing ball 
in $H^{2\beta} \times H^{2\alpha}$ in subsection \ref{sec4.2}, and that there are a finite number of determining modes on the global attractor in subsection \ref{sec4.3}.

\section{Notations and preliminaries} \label{sec2}
\subsection{Notations and function spaces}
Here and throughout this article, we will not distinguish the notations for vector and scalar function spaces whenever they are self-evident from the context. 
Let $L^{p}(\Omega)$ $(1 \leq p \leq \infty)$ be the classical Lebesgue space with norm $|| \cdot ||_{L^{p}}$ and $\mathcal{C}([0,T]; X)$ be the space of all continuous functions from the interval $[0,T]$ to some normed space $X$. We denote by $L^{p}(0,T; X)$ $(1 \leq p \leq \infty)$ the space of all measurable functions $u : [0,T] \rightarrow X$ with the norm
\[
\aiminnorm{u}^p_{L^{p}(0,T;X)}= \int_0^T \aiminnorm{u}^p_{X} {\mathrm{d}t},
\qquad \aiminnorm{u}_{L^{\infty}(0,T;X)}=  \mathrm{ess}\sup\limits_{t \in [0,T]}
\aiminnorm{u}_{X}.
\]
For $f \in L^1(\Omega)$ and $k = (k_1, k_2) \in \mathbb{Z}^2$,
the Fourier coefficient $\hat{f}(k)$ of $f$ is defined as
$$\widehat{f}(k)= \frac{1}{(2\pi)^2} \int_{\Omega} f(x)e^{-ik \cdot x} {\mathrm{d}x}. $$
We denote the square root of the Laplacian $(-\Delta)^{\frac{1}{2}}$ by $\Lambda$ and we have
$$\widehat{\Lambda f}(k)= |k| \widehat{f}(k),$$
where $\aiminabs{k}=\sqrt{k_1^2+k_2^2}$. More generally, for
$s \in \mathbb{R}$, the fractional Laplacian $\Lambda^{s}f$ can be defined by the Fourier series
\[
\Lambda^{s}f := \sum_{k \in \mathbb{Z}^2} \aiminabs{k}^{s} \widehat{f}(k)e^{ik \cdot x}.
\]
We denote by $H^{s}(\Omega)$ the space of all the functions $f$
of mean zero with $\aiminnorm{f}_{H^s} < \infty$ where the norm $\aiminnorm{\cdot}_{H^s}$ is defined as 
$$\aiminnorm{f}_{H^s}^2= \aiminnorm{\Lambda^s f}_{L^2}^2 = \sum_{k \in \mathbb{Z}^2}
|k|^{2s} |\widehat{f}(k)|^2.$$
For $1 \leq p \leq \infty$ and $s \in \mathbb{R}$, the space $H^{s,p}(\Omega)$ consists of the functions $f$
such that $f= \Lambda^{-s} g$ for some $g \in L^p(\Omega)$. The $H^{s,p}$-norm of $f$ is defined by
\[
\aiminnorm{f}_{H^{s,p}} = \aiminnorm{\Lambda^{s} f}_{L^{p}}.
\]
By the classic spectral theory of compact operators, we denote by $\{\lambda_j \}_{j=1}^{\infty} (0<\lambda_1=1 \leq \lambda_2 \leq\lambda_3 \leq \cdots) $
the eigenvalues of the operator $\Lambda$, which are repeated according to their multiplicities, arranged in the non-decreasing order corresponding to the eigenfunctions $\{\omega_j \}_{j=1}^{\infty}$.
For the sake of simplicity, we use $\aiminnorm{\cdot}$ to stand for the $L^2$-norm and write $L^p$, $H^s$, and $H^{s,p}$ to stand for the space $L^p(\Omega)$, $H^s(\Omega)$ and $H^{s,p}(\Omega)$ respectively for $1 \leq p \leq \infty$ and $s \in \mathbb{R}$.
\begin{remark}
Since the first eigenvalue $\lambda_1$ of the operator $\Lambda$ is 1, we could deduce that the constant in Poincar\'e inequality is also 1, that is if $s_1 \leq s_2$, then
\[
\aiminnorm{\Lambda^{s_1} g} \leq \aiminnorm{\Lambda^{s_2} g}, \qquad \forall g \in H^{s_2}.
\] 
\end{remark}
\subsection{Some preliminary results}
We first recall the sharp fractional Sobolev inequality. See \cite{JN14}.
\begin{lemma}[The Sobolev inequality]\label{lemma2.0}
For $0 < s <1$ and $p=\frac{2}{1-s}$, we have
\[
\left(\int_{\Omega} |u(x)|^p \mathrm{d}x \right)^{\frac{2}{p}} \leq C_{s}\aiminnorm{u(x)}^2_{H^s}, \qquad \text{for all $u \in H^s(\Omega).$}
\]
where the best constant $C_s$ is given by 
$C_s= \frac{\Gamma(1-s)}{(4\pi)^s(\pi)^{\frac{s}{2}}\Gamma(1+s)}$.
\end{lemma}

Next, we recall the interpolation inequality and Uniform Gronwall Lemma, which are used frequently in this article. For the proofs of the interpolation inequality and Uniform Gronwall Lemma, one can refer to \cite{Tem88}. 
\begin{lemma}[The interpolation inequality]\label{lemma2.0.1}
For any $s_1\leq s\leq s_2$ and $g \in H^{s_2}$, we have
\[
\aiminnorm{ \Lambda^s g} \leq \aiminnorm{ \Lambda^{s_1} g}^{\delta} \aiminnorm{ \Lambda^{s_2} g}^{1-\delta},
\]
where $s= \delta s_1 +  (1-\delta) s_2$ for some $0 \leq \delta \leq 1$.
\end{lemma}
\begin{lemma}[Uniform Gronwall Lemma]\label{lemma 2.1}
Let $g$, $h$ and $y$ be non-negative locally integrable functions on $(t_0, +\infty)$ such that
$$\frac{{\mathrm{d}y}(t)}{{\mathrm{d}t}} \leq g(t)y(t)+h(t), \qquad \forall t \geq t_0,$$
and
$$\int_{t}^{t+r} g(s) {\mathrm{d}s} \leq a_1, \qquad \int_{t}^{t+r} h(s) {\mathrm{d}s} \leq a_2, \qquad \int_{t}^{t+r} y(s) {\mathrm{d}s} \leq a_3, \qquad  \forall t \geq t_0,$$
where $r, a_1, a_2$ and $a_3$ are positive constants. Then
$$y(t+r) \leq \left(\frac{a_3}{r}+a_2\right)e^{a_1}, \qquad  \forall t \geq t_0.$$
\end{lemma}

We will use the following Kate-Ponce and commutator inequalities from \cite{KP88}, see also \cite{WU02, Ju05}.
\begin{lemma} \label{lem2.4}
Suppose that $g, h \in C^{\infty}_c(\Omega)$,  then
\begin{equation} \label{eqn2}
\aiminnorm{\Lambda^{s}(g h)} \leq
C(\aiminnorm{\Lambda^{s} g}_{L^{p_1}}\aiminnorm{h}_{L^{p_2}}+
\aiminnorm{\Lambda^{s}h}_{L^{q_1}}\aiminnorm{g}_{L^{q_2}}),
\end{equation}
where $s >0$, $2 \leq p_1, p_2, q_1, q_2 \leq \infty$ and $1/2=1/p_1+1/{p_2}= 1/{q_1}+1/{q_2}$.
\end{lemma}
\begin{lemma} \label{lem2.5}
Suppose that $\boldsymbol g \in (C^{\infty}_c(\Omega))^2$ and $h \in C^{\infty}_c(\Omega)$,  then
\begin{equation} \label{eqn3}
\aiminnorm{\Lambda^{s}(\boldsymbol g \cdot \nabla h)- \boldsymbol g \cdot (\Lambda^{s}\nabla h)} \leq C(\aiminnorm{\nabla \boldsymbol g}_{L^{p_1}}\aiminnorm{\Lambda^{s} h}_{L^{p_2}}+\aiminnorm{\Lambda^{s} \boldsymbol g}_{L^{q_1}}\aiminnorm{\nabla h}_{L^{q_2}}),
\end{equation}
where $s >0$, $2 < p_1, p_2, q_1, q_2 \leq \infty$ and $1/2=1/p_1+1/{p_2}= 1/{q_1}+1/{q_2}$.
\end{lemma}
\begin{remark}
We remark that the inequalities \eqref{eqn2} and \eqref{eqn3} in Lemmas \ref{lem2.4} and \ref{lem2.5} are also valid for those $g$ (or $\boldsymbol g$) and $h$ belonging to certain Sobolev spaces which make the right-hand sides of \eqref{eqn2} and \eqref{eqn3} finite.
\end{remark}
We now recall the following existence and uniqueness results from \cite{HH15} for the 2D Boussinesq problem.
\begin{theorem} \label{thm2.0}
Let $$\dot{H_0} = \left\{\theta \in L^2: \int_{\Omega} \theta {\mathrm{d}x} =0 \right\},$$ and $$ \dot{H_1} = \left\{\boldsymbol u \in L^2: \nabla \cdot \boldsymbol u =0,
\int_{\Omega} u_1 {\mathrm{d}x} =\int_{\Omega} u_2 {\mathrm{d}x}=0 \right\}.$$ Suppose $f \in H^{-\beta}$ and $(\theta_0, \boldsymbol u_0) \in \dot{H_0} \times \dot{H_1}$. Then, for any $T > 0$, there exists at least one weak solution $(\theta(t), \boldsymbol u(t))$ of the 2D Boussinesq equations \eqref{eq1.1.1} in the sense of distribution.
Moreover, $\theta \in L^{\infty}(0,T; \dot{H_0}) \cap L^{2}(0,T; H^{\beta})$ and $\boldsymbol u \in L^{\infty}(0,T; \dot{H_1}) \cap L^{2}(0,T; H^{\alpha})$.
Furthermore, if we assume that $s_1$, $s_2$ satisfy 
\eqref{cond1} and \eqref{cond2}, 
$(\theta_0, \boldsymbol u_0)\in H^{s_1} \times H^{s_2}$ and $f \in H^{s_1-\beta} \cap L^{p_0}$, where 
\begin{equation} \label{eq2f}
\begin{split}
& r_0=\begin{cases}
s_1, \quad & 2\mathrm{max}\{1-\alpha, 1-\beta\} < s_1 < 1,\\
\text{any number in} \hspace{2pt} (2\mathrm{max} \{1-\alpha, 1-\beta\}, 1),&\qquad s_1\geq 1, \\
\end{cases} \\
& p_0 = \frac{2}{1-r_0}.
\end{split}
\end{equation}
Then for any $T>0$, the Boussinesq system \eqref{eq1.1.1}-\eqref{eq1.1.2} has a unique strong solution $(\boldsymbol u,  \theta)$ satisfying
\begin{equation}\begin{split}
	(\theta, \boldsymbol u) &\in \,\mathcal C([0, T], H^{s_1})
\times	 \mathcal C([0, T], H^{s_2}),\\
	(\theta_t, \boldsymbol u_t) &\in L^2(0, T; H^{s_1-\beta})\times L^2(0, T; H^{s_2-\alpha}).\\
\end{split}\end{equation}
\end{theorem}
It was also proved in \cite{HH15} that the 2D Boussinesq system has a finite-dimensional global attractor.
\begin{theorem}[Existence of a global attractor] \label{thm2.1} 
Assume that $\nu >0$, $\kappa >0$, $s_1$, $s_2$ satisfy \eqref{cond1} and \eqref{cond2},  
and $f \in H^{s_1-\beta} \cap L^{p_0}$ where $p_0$ is defined in \eqref{eq2f}. Then the solution operator $\{S(t)\}_{t\geq 0}$ of the 2D Boussinesq system: $S(t)(\theta_0, \boldsymbol u_0)=(\theta(t),\boldsymbol u(t))$ defines a semigroup in the space $H^{s_1} \times H^{s_2}$ for all $t\in\mathbb R_+$. Moreover, the following statements are valid:
	\begin{enumerate}
		\item for any $(\theta_0, \boldsymbol u_0)\in H^{s_1} \times H^{s_2}$, $t\mapsto S(t)(\theta_0, \boldsymbol u_0)$ is a continuous function from $\mathbb R_+$ into $H^{s_1} \times H^{s_2}$;
		\item for any fixed $t>0$, $S(t)$ is a continuous and compact map in $H^{s_1} \times H^{s_2}$;
		\item $\{S(t)\}_{t\geq 0}$ possesses a global attractor $\mathcal A$ in the space $H^{s_1} \times H^{s_2}$. The global attractor $\mathcal A$ is compact and connected in $H^{s_1} \times H^{s_2}$ and is the maximal bounded attractor and the minimal invariant set in $H^{s_1} \times H^{s_2}$ in the sense of the set inclusion relation.		
		\end{enumerate}
\end{theorem}

\section{Dimensions of the global attractor} \label{sec3}
The aim of this section is to prove that the global attractor $\mathcal{A}$ in Theorem~\ref{thm2.1} has finite Hausdorff and fractal dimensions and we are going to utilize the Ladyzhenskaya squeezing property to estimate the dimension of the global attract $\mathcal{A}$. An alternative approach is to use Lyapunov exponents to estimate the Hausdorff and fractal dimensions of the global attractor $\mathcal{A}$, see \cite{Tem88} and \cite{Lad91} for details. The main result in this section is the following. 
\begin{theorem} \label{th2}
Under the assumptions of Theorem~\ref{thm2.1}, the global attractor $\mathcal{A}$ has finite Hausdorff and fractal dimensions measured in the $H^{s_1} \times H^{s_2}$ space.
\end{theorem}
In order to prove Theorem~\ref{th2}, we first recall the following result from Ladyzhenskaya, see \cite{Lad90} and \cite{JT15}.
\begin{theorem} \label{th1}
Let $X$ be a Hilbert space, $S: X \rightarrow X$ be a map and 
$\mathcal{A} \subset X$ be a compact set such that $S(\mathcal{A})=\mathcal{A}$. Suppose that there exist 
$l \in [1, +\infty)$ and $\delta \in (0,1)$, such that $\forall a_1, a_2 \in \mathcal{A}$,
\[
\aiminnorm{S(a_1)-S(a_2)}_{X} \leq l\aiminnorm{a_1-a_2}_{X},
\]
\[
\aiminnorm{Q_N[S(a_1)-S(a_2)]}_{X} \leq \delta\aiminnorm{a_1-a_2}_{X},
\]
where $Q_N$ is the projection in $X$ onto some subspace $(X_N)^{\perp}$ of co-dimension $N \in \mathbb{N}$. Then,
\[
d_{H}(\mathcal{A}) \leq d_{F}(\mathcal{A}) \leq
N\frac{\ln(\frac{8G_{a}^2l^2}{1-\delta^2})}{\ln(\frac{2}{1+\delta^2})},
\]
where $d_{H}(\mathcal{A})$ and $d_{F}(\mathcal{A})$ are the Hausdorff and fractal dimensions of $\mathcal{A}$ respectively and $G_a$ is the Gauss constant:
\[
G_a = \frac{2}{\pi}
\int_0^1 \frac{dx}{\sqrt{1-x^4}} = 0.8346268....
\]
\end{theorem}
\begin{proof}[Proof of Theorem~\ref{th2}]
Suppose $(\theta_1, \boldsymbol u_1, \pi_1)$, $(\theta_2, \boldsymbol u_2, \pi_2)$ are two strong solutions of 2D Boussinesq systems \eqref{eq1.1.1} with two initial data $(\theta_1^0, \boldsymbol u_1^0), (\theta_2^0, \boldsymbol u_2^0) \in \mathcal{A}$ respectively. Let $\eta = \theta_1- \theta_2$ and $\boldsymbol w = \boldsymbol u_1- \boldsymbol u_2$. Then $(\eta, \boldsymbol w)$ satisfies the following equations:
\begin{equation}\begin{cases}\label{eq4.1}
\partial_t\boldsymbol w + \boldsymbol u_1\cdot \nabla \boldsymbol w + \boldsymbol w \cdot \nabla \boldsymbol u_2 + \nu(-\Delta)^\alpha \boldsymbol w= -\nabla(\pi_1-\pi_2)+
\eta \boldsymbol e_2, \\
\partial_t \eta +\boldsymbol w\cdot \nabla \theta_1+ \boldsymbol u_2 \cdot \nabla \eta + \kappa(-\Delta)^{\beta}\eta = 0.
\end{cases}
\end{equation}
Let $P_m$ be the projection onto the subspace spanned by the first $m$ eigenvectors of the operator $\Lambda$ associated with the eigenvalues
$\lambda_1, \cdots, \lambda_m$ and set $Q_m = I - P_m$. Multiplying by $\Lambda^{2s_2} Q_m \boldsymbol w$, $\Lambda^{2s_1} Q_m \eta$ on the equations~\eqref{eq4.1}$_1$ and \eqref{eq4.1}$_2$ respectively, 
and taking the inner product in $L^2$, we obtain
\begin{equation}\begin{cases}\label{eq4.2}
\frac{1}{2} \frac{\mathrm{d}}{\mathrm{d}t} \aiminnorm{\Lambda^{s_2} Q_m\boldsymbol w}^2 + \nu \aiminnorm{\Lambda^{s_2+\alpha}Q_m \boldsymbol w}^2= - \aimininner{\boldsymbol u_1 \cdot \nabla \boldsymbol w}{\Lambda^{2s_2} Q_m \boldsymbol w} - \aimininner{\boldsymbol w \cdot \nabla \boldsymbol u_2}{\Lambda^{2s_2} Q_m \boldsymbol w}  \\
\hspace{180pt}+ \aimininner{\eta \boldsymbol e_2}{\Lambda^{2s_2} Q_m \boldsymbol w}, \\
\frac{1}{2} \frac{\mathrm{d}}{\mathrm{d}t} \aiminnorm{\Lambda^{s_1} Q_m \eta}^2 + \kappa \aiminnorm{\Lambda^{s_1+\beta} Q_m \eta}^2= 
- \aimininner{\boldsymbol w \cdot \nabla \theta_1}{\Lambda^{2s_1}Q_m \eta} -\aimininner{\boldsymbol u_2 \cdot \nabla \eta}{\Lambda^{2s_1}Q_m \eta}.
\end{cases}\end{equation}
As a preliminary, we see from the condition \eqref{cond1} that
\[
s_1 > 2\mathrm{max} \{1-\alpha, 1-\beta \} \geq 1-\alpha+1-\beta
= 2- \alpha- \beta.
\]
Hence, we could fix an $\alpha_1 \in (1/2, \alpha)$ such that
\[
s_1 \geq 2-\alpha_1-\beta,
\]
and we also fix a $\beta_1 \in (1/2, \beta)$. Since $s_2 \geq 1$, we have
\[
s_2 \geq 1 > 2-\alpha_1-\alpha, \qquad s_2 \geq 1 > 2-\beta_1-\beta. 
\]
Therefore, by the Sobolev embedding theorem, we have
\begin{equation} \label{t1}
H^{s_1} \subset \subset H^{2-\alpha_1-\beta} \subset \subset L^{\frac{2}{\alpha_1+\beta-1}},
\end{equation}
and
\begin{equation} \label{t2}
H^{s_2} \subset \subset H^{2-\alpha_1-\alpha} \subset \subset L^{\frac{2}{\alpha_1+\alpha-1}}, \qquad
H^{s_2} \subset \subset H^{2-\beta_1-\beta} \subset \subset L^{\frac{2}{\beta+\beta_1-1}}.
\end{equation}

We now estimate the term $\aimininner{\boldsymbol u_1 \cdot \nabla \boldsymbol w}{\Lambda^{2s_2} Q_m \boldsymbol w}$. Since $\boldsymbol w$ is divergence free, 
\begin{equation} \label{eq4.3}
\begin{split}
|\aimininner{\boldsymbol u_1 \cdot \nabla \boldsymbol w}{\Lambda^{2s_2} Q_m \boldsymbol w}|
& = \aimininner{\Lambda^{s_2-\alpha}(\boldsymbol u_1 \cdot \nabla \boldsymbol w)}{\Lambda^{s_2+\alpha} Q_m \boldsymbol w} \\
& \leq \aiminnorm{\Lambda^{s_2+1-\alpha}(\boldsymbol u_1 \otimes \boldsymbol w)}\aiminnorm{\Lambda^{s_2+\alpha} Q_m \boldsymbol w}.
\end{split}
\end{equation}
Let $p, q > 2$ such that $1/p+1/q=1/2$ and we choose 
\[
r= 2-\alpha-\alpha_1, \qquad p=\frac{2}{r}=\frac{2}{2-\alpha-\alpha_1} \qquad q= \frac{2}{1-r}=\frac{2}{\alpha+\alpha_1-1}.
\] 
Applying Lemma~\ref{lemma2.0} and Lemma~\ref{lem2.4}, we have
\begin{equation} \label{eq4.4}
\begin{split}
\aiminnorm{\Lambda^{s_2+1-\alpha}(\boldsymbol u_1 \otimes \boldsymbol w)}
& \leq C(\aiminnorm{\Lambda^{s_2+1-\alpha} \boldsymbol u_1}_{L^p}\aiminnorm{\boldsymbol w}_{L^q}+\aiminnorm{\Lambda^{s_2+1-\alpha} \boldsymbol w}_{L^p}\aiminnorm{\boldsymbol u_1}_{L^q}) \\
& \leq C(\aiminnorm{\Lambda^{s_2+2-\alpha-r} \boldsymbol u_1}
\aiminnorm{\boldsymbol w}_{L^{\frac{2}{1-r}}}+\aiminnorm{\Lambda^{s_2+2-\alpha-r} \boldsymbol w}\aiminnorm{\boldsymbol u_1}_{L^{\frac{2}{1-r}}}) \\
& \leq C(\aiminnorm{\Lambda^{s_2+\alpha_1} \boldsymbol u_1}
\aiminnorm{\Lambda^{s_2} \boldsymbol w}+\aiminnorm{\Lambda^{s_2+\alpha_1} \boldsymbol w}\aiminnorm{\Lambda^{s_2} \boldsymbol u_1}),
\end{split}
\end{equation}
where we used \eqref{t2} for the last inequality. Therefore, by the interpolation inequality in Lemma~\ref{lemma2.0.1} and the Cauchy-Schwarz inequality, we have
\begin{equation} \label{eq4.5}
\begin{split}
|\aimininner{\boldsymbol u_1 \cdot \nabla \boldsymbol w}{\Lambda^{2s_2} Q_m \boldsymbol w}| &\leq C\aiminnorm{\Lambda^{s_2+\alpha} \boldsymbol u_1}^{\frac{\alpha_1}{\alpha}}
\aiminnorm{\Lambda^{s_2} \boldsymbol u_1}^{1-\frac{\alpha_1}{\alpha}}
\aiminnorm{\Lambda^{s_2} \boldsymbol w}\aiminnorm{\Lambda^{s_2+\alpha} Q_m\boldsymbol w}\\
& \hspace{5pt}+ C\aiminnorm{\Lambda^{s_2+\alpha} \boldsymbol w}^{\frac{\alpha_1}{\alpha}}
\aiminnorm{\Lambda^{s_2} \boldsymbol w}^{1-\frac{\alpha_1}{\alpha}}\aiminnorm{\Lambda^{s_2} \boldsymbol u_1}
\aiminnorm{\Lambda^{s_2+\alpha} Q_m\boldsymbol w} \\
& \leq C\aiminnorm{\Lambda^{s_2+\alpha} \boldsymbol u_1}^{\frac{2\alpha_1}{\alpha}}
\aiminnorm{\Lambda^{s_2} \boldsymbol u_1}^{\frac{2(\alpha-\alpha_1)}{\alpha}}
\aiminnorm{\Lambda^{s_2} \boldsymbol w}^2 \\
& \hspace{5pt}+ C\aiminnorm{\Lambda^{s_2+\alpha} \boldsymbol w}^{\frac{2\alpha_1}{\alpha}}
\aiminnorm{\Lambda^{s_2} \boldsymbol w}^{\frac{2(\alpha-\alpha_1)}{\alpha}}
\aiminnorm{\Lambda^{s_2} \boldsymbol u_1}^2+
\frac{\nu}{6}\aiminnorm{\Lambda^{s_2+\alpha} Q_m \boldsymbol w}^2.
\end{split}
\end{equation}
Similar to \eqref{eq4.5}, we have the estimate
\begin{equation} \label{eq4.6}
\begin{split}
|\aimininner{\boldsymbol w \cdot \nabla \boldsymbol u_2}{\Lambda^{2s_2} Q_m \boldsymbol w}|  
& \leq C\aiminnorm{\Lambda^{s_2+\alpha} \boldsymbol w}^{\frac{2\alpha_1}{\alpha}}
\aiminnorm{\Lambda^{s_2} \boldsymbol w}^{\frac{2(\alpha-\alpha_1)}{\alpha}}
\aiminnorm{\Lambda^{s_2} \boldsymbol u_2}^2 \\
& + C\aiminnorm{\Lambda^{s_2+\alpha} \boldsymbol u_2}^{\frac{2\alpha_1}{\alpha}}
\aiminnorm{\Lambda^{s_2} \boldsymbol u_2}^{\frac{2(\alpha-\alpha_1)}{\alpha}}
\aiminnorm{\Lambda^{s_2} \boldsymbol w}^2+
\frac{\nu}{6}\aiminnorm{\Lambda^{s_2+\alpha} Q_m \boldsymbol w}^2.
\end{split}
\end{equation}
Next, we estimate the term $\aimininner{\boldsymbol w \cdot \nabla \theta_1}{\lambda^{2s_1}Q_m \eta}$. Since $\boldsymbol w$ is divergence free, 
\begin{equation}  \nonumber
\begin{split}
|\aimininner{\boldsymbol w \cdot \nabla \theta_1}{\Lambda^{2s_1}Q_m \eta}|
& = \aimininner{\Lambda^{s_1-\beta}(\boldsymbol w \cdot \nabla \theta_1)}{\Lambda^{s_1+\beta} Q_m \eta} \\
& \leq \aiminnorm{\Lambda^{s_1+1-\beta}(\boldsymbol w \cdot \theta_1)}\aiminnorm{\Lambda^{s_1+\beta} Q_m \eta}.
\end{split}
\end{equation}
Let $p_1, q_1,p_2, q_2 > 2$ such that $1/p_1+1/q_1= 1/p_2+1/q_2=1/2$ and we choose 
\[
r_1= 2-\alpha_1-\beta, \qquad p_1=\frac{2}{r_1}=\frac{2}{2-\alpha_1-\beta}, \qquad q_1= \frac{2}{1-r_1}=\frac{2}{\alpha_1+\beta-1},
\]
and
\[
r_2= 2-\beta-\beta_1, \qquad p_2=\frac{2}{r_2}=\frac{2}{2-\beta-\beta_1}, \qquad q_2= \frac{2}{1-r_2}=\frac{2}{\beta+\beta_1-1}. 
\]
Applying Lemma~\ref{lemma2.0} and Lemma~\ref{lem2.4}, since $s_1 \leq s_2$, we have
\begin{equation} \label{eq4.8}
\begin{split}
\aiminnorm{\Lambda^{s_1+1-\beta}(\boldsymbol w \cdot \theta_1)}
& \leq C(\aiminnorm{\Lambda^{s_1+1-\beta} \boldsymbol w}_{L^{p_1}}\aiminnorm{\theta_1}_{L^{q_1}}+\aiminnorm{\Lambda^{s_1+1-\beta} \theta_1}_{L^{p_2}}\aiminnorm{\boldsymbol w}_{L^{q_2}}) \\
& \leq C(\aiminnorm{\Lambda^{s_1+2-\beta-r_1} \boldsymbol w}
\aiminnorm{\theta_1}_{L^{\frac{2}{1-r_1}}}+\aiminnorm{\Lambda^{s_1+2-\beta-r_2} \theta_1}\aiminnorm{\boldsymbol w}_{L^{\frac{2}{1-r_2}}}) \\
& \leq C(\aiminnorm{\Lambda^{s_2+\alpha_1} \boldsymbol w}
\aiminnorm{\Lambda^{s_1} \theta_1}+\aiminnorm{\Lambda^{s_1+\beta_1} \theta_1}\aiminnorm{\Lambda^{s_2} \boldsymbol w}),
\end{split}
\end{equation}
where we used \eqref{t1}-\eqref{t2} for the last inequality. Thus, applying the interpolation inequality in Lemma~\ref{lemma2.0.1} and the Cauchy-Schwarz inequality, we obtain
\begin{equation} \label{eq4.9}
\begin{split}
|\aimininner{\boldsymbol w \cdot \nabla \theta_1}{\Lambda^{2s_1}Q_m \eta}|
&\leq C(\aiminnorm{\Lambda^{s_2+\alpha} \boldsymbol w}^{\frac{\alpha_1}{\alpha}}
\aiminnorm{\Lambda^{s_2} \boldsymbol w}^{1-\frac{\alpha_1}{\alpha}}
\aiminnorm{\Lambda^{s_1} \theta_1}\aiminnorm{\Lambda^{s_1+\beta} Q_m\eta} \\
& \hspace{5pt} +C\aiminnorm{\Lambda^{s_1+\beta} \theta_1}^{\frac{\beta_1}{\beta}}
\aiminnorm{\Lambda^{s_1} \theta_1}^{1-\frac{\beta_1}{\beta}}\aiminnorm{\Lambda^{s_2} \boldsymbol w}\aiminnorm{\Lambda^{s_1+\beta} Q_m\eta} \\
& \leq C\aiminnorm{\Lambda^{s_2+\alpha} \boldsymbol w}^{\frac{2\alpha_1}{\alpha}}
\aiminnorm{\Lambda^{s_2} \boldsymbol w}^{\frac{2(\alpha-\alpha_1)}{\alpha}}
\aiminnorm{\Lambda^{s_1} \theta_1}^2\\
& \hspace{5pt}+ C\aiminnorm{\Lambda^{s_1+\beta} \theta_1}^{\frac{2\beta_1}{\beta}}
\aiminnorm{\Lambda^{s_1} \theta_1}^{\frac{2(\beta-\beta_1)}{\beta}}
\aiminnorm{\Lambda^{s_2} \boldsymbol w}^2+
\frac{\kappa}{6}\aiminnorm{\Lambda^{s_1+\beta} Q_m \eta}^2.
\end{split}
\end{equation}
Similar to \eqref{eq4.9}, we have the estimate
\begin{equation} \label{eq4.10}
\begin{split}
|\aimininner{\boldsymbol u_2 \cdot \nabla \eta}{\Lambda^{2s_1}Q_m \eta}| & \leq C\aiminnorm{\Lambda^{s_2+\alpha} \boldsymbol u_2}^{\frac{2\alpha_1}{\alpha}}
\aiminnorm{\Lambda^{s_2} \boldsymbol u_2}^{\frac{2(\alpha-\alpha_1)}{\alpha}}
\aiminnorm{\Lambda^{s_1} \eta}^2 \\
& + C\aiminnorm{\Lambda^{s_1+\beta} \eta}^{\frac{2\beta_1}{\beta}}
\aiminnorm{\Lambda^{s_1} \eta}^{\frac{2(\beta-\beta_1)}{\beta}}
\aiminnorm{\Lambda^{s_2} \boldsymbol u_2}^2+
\frac{\kappa}{6}\aiminnorm{\Lambda^{s_1+\beta} Q_m \eta}^2.\end{split}
\end{equation}
Finally, applying the interpolation inequality in Lemma~\ref{lemma2.0.1} and the Cauchy-Schwarz inequality, we have
\begin{equation} \label{eq4.11}
\begin{split}
\aiminabs{\aimininner{\eta \boldsymbol e_2}{\Lambda^{2s_2} Q_m\boldsymbol w}}
& = \aiminabs{\aimininner{\Lambda^{s_1+\beta}Q_m \eta \boldsymbol e_2}{\Lambda^{2s_2-s_1-\beta} Q_m\boldsymbol w}} \\
&\leq \aiminnorm{\Lambda^{s_1+\beta}Q_m\eta}\aiminnorm{\Lambda^{2s_2-s_1-\beta}Q_m\boldsymbol w} \\
& \leq \aiminnorm{\Lambda^{s_1+\beta}Q_m\eta}\aiminnorm{\Lambda^{s_2+\alpha}Q_m\boldsymbol w}^{1-r^*}
\aiminnorm{\Lambda^{s_2} Q_m\boldsymbol w}^{r^*} \\
& \leq \frac{\kappa}{6} \aiminnorm{\Lambda^{s_1+\beta}Q_m\eta}^2+ \frac{\nu}{6} \aiminnorm{\Lambda^{s_2+\alpha}Q_m\boldsymbol w}^2 + C_1\aiminnorm{\Lambda^{s_2} Q_m\boldsymbol w}^2,
\end{split}
\end{equation}
where $\tilde{r} = (s_1+\alpha+\beta-s_2)/\alpha > 0$ and 
$C_1= \frac{1}{\kappa^{\tilde{r}}\nu^{\frac{1}{\tilde{r}}-1}}$.

It was shown in \cite{HH15} that when the solutions $(\theta_i, \boldsymbol u_i) \in \mathcal{A} \subset H^{s_1} \times H^{s_2}$ for $i = 1, 2$, $\aiminnorm{\Lambda^{s_1+\beta} \theta_i}$ and $\aiminnorm{\Lambda^{s_2+\alpha} \boldsymbol u_i}$ are uniformly bounded independent of $t$ for $i = 1, 2$. Therefore, summing \eqref{eq4.2},\eqref{eq4.5}-\eqref{eq4.6}, and \eqref{eq4.9}-\eqref{eq4.11} together, we obtain
\begin{equation}\label{eq4.13}
\begin{split}
&\frac{\mathrm{d}}{\mathrm{d}t} (\aiminnorm{\Lambda^{s_2} Q_m\boldsymbol w}^2 
+ \aiminnorm{\Lambda^{s_1}Q_m \eta}^2) + \nu \aiminnorm{\Lambda^{s_2+\alpha}Q_m \boldsymbol w}^2
+\kappa \aiminnorm{\Lambda^{s_1+\beta} Q_m \eta}^2 \\
\leq  & \hspace{3pt} C\aiminnorm{\Lambda^{s_2+\alpha} \boldsymbol w}^{\frac{2\alpha_1}{\alpha}}\aiminnorm{\Lambda^{s_2} \boldsymbol w}^{\frac{2(\alpha-\alpha_1)}{\alpha}} 
C\aiminnorm{\Lambda^{s_1+\beta} \eta}^{\frac{2\beta_1}{\beta}}\aiminnorm{\Lambda^{s_1} \eta}^{\frac{2(\beta-\beta_1)}{\beta}} \\ 
&+  C(\aiminnorm{\Lambda^{s_2} \boldsymbol w}^2+ 
\aiminnorm{\Lambda^{s_1} \eta}^2)+
C_1\aiminnorm{\Lambda^{s_2} Q_m\boldsymbol w}^2.
\end{split}
\end{equation}
By Poincar\'e's inequality, we have
\[
\lambda_m^{2\alpha} \aiminnorm{\Lambda^{s_2} Q_m \boldsymbol w}^2
\leq \aiminnorm{\Lambda^{s_2+\alpha} Q_m \boldsymbol w}^2,
\qquad \mathrm{and} \qquad
\lambda_m^{2\beta} \aiminnorm{\Lambda^{s_1} Q_m \eta}^2
\leq \aiminnorm{\Lambda^{s_1+\beta} Q_m \eta}^2.
\]
We denote 
\[
y(t) := \aiminnorm{\Lambda^{s_2} \boldsymbol w(t)}^2+ \aiminnorm{\Lambda^{s_1} \eta(t)}^2, \qquad
z(t) := \aiminnorm{\Lambda^{s_2} Q_m\boldsymbol w(t)}^2+ \aiminnorm{\Lambda^{s_1} Q_m\eta(t)}^2,
\]
and let $\rho_m = \frac{1}{2}\min \{ \nu\lambda_m^{\alpha},  
\kappa\lambda_m^{\beta}\}$. Then we can choose $m$ large enough, so that $C_1 \leq \rho_m$. Hence, it follows from \eqref{eq4.13} that
\begin{equation}\label{eq4.14}
\begin{split}
z'(t)+ \rho_m z(t) \leq   
Cy(t) + C\aiminnorm{\Lambda^{s_2+\alpha} \boldsymbol w}^{\frac{2\alpha_1}{\alpha}}\aiminnorm{\Lambda^{s_2} \boldsymbol w}^{\frac{2(\alpha-\alpha_1)}{\alpha}}
+C\aiminnorm{\Lambda^{s_1+\beta} \eta}^{\frac{2\beta_1}{\beta}}\aiminnorm{\Lambda^{s_1} \eta}^{\frac{2(\beta-\beta_1)}{\beta}}.
\end{split}
\end{equation}
Now, integrating \eqref{eq4.14} with respect to $t \in [0, T]$, we have
\begin{equation} \label{eq4.15}
\begin{split}
z(T) & \leq e^{-\rho_mT}z(0)+ Ce^{-\rho_mT}\int_0^T e^{\rho_mt}
y(t) {\mathrm{d}t} \\
& \hspace{7pt} + Ce^{-\rho_mT}\int_0^T e^{\rho_mt}\aiminnorm{\Lambda^{s_2+\alpha} \boldsymbol w}^{\frac{2\alpha_1}{\alpha}}\aiminnorm{\Lambda^{s_2} \boldsymbol w}^{\frac{2(\alpha-\alpha_1)}{\alpha}} {\mathrm{d}t} \\
& \hspace{7pt} + Ce^{-\rho_mT}\int_0^T e^{\rho_mt}\aiminnorm{\Lambda^{s_1+\beta} \eta}^{\frac{2\beta_1}{\beta}}\aiminnorm{\Lambda^{s_1} \eta}^{\frac{2(\beta-\beta_1)}{\beta}} {\mathrm{d}t} \\
& = : I_1+I_2+I_3+I_4.
\end{split}
\end{equation}
First, we notice that 
\begin{equation} \label{eq4.15}
I_1 \leq e^{-\rho_mT}y(0).
\end{equation}
Next, we recall the results from \cite[Section 4.2]{HH15} that if $(\theta_i, \boldsymbol u_i)$ are two strong solutions in $\mathcal{A} \subset H^{s_1} \times H^{s_2}$ for $i=1,2$, then for all $t \geq 0$, 
\[
\begin{split}
&\aiminnorm{\Lambda^{s_2}\boldsymbol w(t)}^2+\aiminnorm{\Lambda^{s_1}\eta(t)}^2 
+\nu \int_0^t \aiminnorm{\Lambda^{s_2+\alpha} \boldsymbol w}^2 {\mathrm{d}s} +\kappa \int_0^t \aiminnorm{\Lambda^{s_1+\beta} \eta}^2 {\mathrm{d}s}\\
\leq \hspace{3pt} &C(\aiminnorm{\Lambda^{s_2}\boldsymbol w(0)}^2+\aiminnorm{\Lambda^{s_1}\eta(0)}^2) \\
&\mathrm{exp} \left\{\int_0^t \aiminnorm{\Lambda^{s_2+\alpha} \boldsymbol u_1(s)}^2+\aiminnorm{\Lambda^{s_2+\alpha} \boldsymbol u_2(s)}^2+\aiminnorm{\Lambda^{s_1+\beta} \theta_2(s)}^2 {\mathrm{d}s}    \right\}.
\end{split}
\]
Thus,
\begin{equation} \label{eq4.16.1}
y(t)+\sigma \int_0^t \aiminnorm{\Lambda^{s_2+\alpha} \boldsymbol w}^2 +\aiminnorm{\Lambda^{s_1+\beta} \eta}^2 {\mathrm{d}s}
\leq y(0)K(t), \qquad \forall t \geq 0,
\end{equation}
where $\sigma = \mathrm{min}\{\nu, \kappa \}$ and 
$$K(t)= C\mathrm{exp} \left\{\int_0^t \aiminnorm{\Lambda^{s_2+\alpha} \boldsymbol u_1(s)}^2+\aiminnorm{\Lambda^{s_2+\alpha} \boldsymbol u_2(s)}^2+\aiminnorm{\Lambda^{s_1+\beta} \theta_2(s)}^2 {\mathrm{d}s}  \right\},$$
which is a positive continuous non-decreasing function on $[0, \infty)$ and independent of the initial data.
Therefore,
\begin{equation} \label{eq4.17}
I_2 \leq  Ce^{-\rho_mT}y(0)K(T)\int_0^T e^{\rho_mt} {\mathrm{d}t}
\leq C\rho_m^{-1}K(T)y(0).
\end{equation}
Finally, 
\begin{equation} \label{eq4.18}
\begin{split}
I_3 &\leq Ce^{-\rho_mT}\int_0^T e^{\rho_mt}y(t)^{\frac{\alpha-\alpha_1}{\alpha}} \aiminnorm{\Lambda^{s_2+\alpha} \boldsymbol w}^{\frac{2\alpha_1}{\alpha}} {\mathrm{d}t} \\
&\leq Ce^{-\rho_mT}y(0)^{\frac{\alpha-\alpha_1}{\alpha}}
K(T)^{\frac{\alpha-\alpha_1}{\alpha}} \int_0^T e^{\rho_mt}  \aiminnorm{\Lambda^{s_2+\alpha} \boldsymbol w}^{\frac{2\alpha_1}{\alpha}} {\mathrm{d}t} \\
&\leq Ce^{-\rho_mT}y(0)^{\frac{\alpha-\alpha_1}{\alpha}}
K(T)^{\frac{\alpha-\alpha_1}{\alpha}}
\left( \int_0^T e^{\rho_mt \cdot \frac{\alpha}{\alpha-\alpha_1}}  {\mathrm{d}t} \right)^{\frac{\alpha-\alpha_1}{\alpha}}
\left(\int_0^T  \aiminnorm{\Lambda^{s_2+\alpha} \boldsymbol w}^2 {\mathrm{d}t} \right)^{\frac{\alpha_1}{\alpha}} \\
&\leq C\rho_m^{-\frac{\alpha-\alpha_1}{\alpha}}K(T)y(0),
\end{split}
\end{equation}
and similarly
\begin{equation} \label{eq4.19}
\begin{split}
I_4 &\leq Ce^{-\rho_mT}\int_0^T e^{\rho_mt}y(t)^{\frac{\beta-\beta_1}{\beta}} \aiminnorm{\Lambda^{s_1+\beta} \eta}^{\frac{2\beta_1}{\beta}} {\mathrm{d}t} \\
&\leq Ce^{-\rho_mT}y(0)^{\frac{\beta-\beta_1}{\beta}}
K(T)^{\frac{\beta-\beta_1}{\beta}} \int_0^T e^{\rho_mt}  \aiminnorm{\Lambda^{s_1+\beta}\eta}^{\frac{2\beta_1}{\beta}} {\mathrm{d}t} \\
&\leq Ce^{-\rho_mT}y(0)^{\frac{\beta-\beta_1}{\beta}}
K(T)^{\frac{\beta-\beta_1}{\beta}}
\left( \int_0^T e^{\rho_mt \cdot \frac{\beta}{\beta-\beta_1}}  {\mathrm{d}t} \right)^{\frac{\beta-\beta_1}{\beta}}
\left(\int_0^T  \aiminnorm{\Lambda^{s_1+\beta} \eta}^2 {\mathrm{d}t} \right)^{\frac{\beta_1}{\beta}} \\
&\leq C\rho_m^{-\frac{\beta-\beta_1}{\beta}}K(T)y(0).
\end{split}
\end{equation}
Summing the estimates in \eqref{eq4.15}, \eqref{eq4.17}, \eqref{eq4.18} and \eqref{eq4.19} together, we have
\begin{equation} \label{eq4.20}
z(T) \leq \left(e^{-\rho_mT}+\rho_m^{-1}K(T)+\rho_m^{-\frac{\alpha-\alpha_1}{\alpha}}K(T)+ \rho_m^{-\frac{\beta-\beta_1}{\beta}}K(T) \right)y(0).
\end{equation}
Therefore, for any fixed $T>0$, $l= K(T) \in [1, \infty)$, 
some $\delta \in (0,1)$, and given two strong solutions $(\theta_i, \boldsymbol u_i) \in \mathcal{A}$, combing the results from \eqref{eq4.16.1} and \eqref{eq4.20}, we can choose $m$ large enough, such that 
\[
y(T) \leq l y(0), \qquad \mathrm{and} \qquad
z(T) \leq \delta y(0).
\]
Hence, Theorem~\ref{th2} immediately follows from Theorem~\ref{th1}.
\end{proof}

\section{Determining modes on the attractor} \label{sec4}
In \cite{HH15}, it was shown that the solution operator $\{S(t)\}_{t\geq 0}$ of the 2D Boussinesq system with periodic boundary condition possess a global attractor which is invariant, compact and connected in the Sobolev space. Then we next consider the concept: determining modes (the number of the first Fourier modes) on the attractor, which was established in \cite{FMRT01}. The theories are based on the dimension analysis and suggest that the long-time behavior of turbulence flows is determined by a finite number of degrees of freedom. 

In this section, we are going to prove that there exists a positive number $m$ large enough, such that if the projections on the space spanned by the first $m$ eigenvectors of the operator $\Lambda$
of two different trajectories on the attractor $\mathcal{A}$ coincide for all $t \in \mathbb{R}$, then these two trajectories actually coincide for all $t \in \mathbb{R}$.

\subsection{The Definition of Determining Modes}
Let us consider two vectors $(\theta_1, \boldsymbol{u}_1)=(\theta_1(x,t), \boldsymbol{u_1}(x,t))$ and $(\theta_2, \boldsymbol{u_2})=(\theta_2(x,t), \boldsymbol{u}_2(x,t))$
satisfying 2D Boussinesq systems with corresponding the forces $f=f(x)$ and $g=g(x)$. More precisely, $(\theta_1, \boldsymbol{u}_1)$ and $(\theta_2, \boldsymbol{u}_2)$ satisfy the equations:
\begin{equation}\begin{cases}\label{eq1.2}
\partial_t\boldsymbol u_1 + \boldsymbol u_1\cdot \nabla \boldsymbol u_1 + \nu(-\Delta)^\alpha \boldsymbol u_1=- \nabla \pi_1 + \theta_1 \boldsymbol e_2, \qquad x \in \Omega, \hspace{2pt} t > 0, \\
\nabla \cdot \boldsymbol u_1=0,  \qquad\qquad\qquad\qquad\qquad\qquad\qquad x \in \Omega, \hspace{2pt} t > 0,\\
\partial_t \theta_1 +\boldsymbol u_1\cdot \nabla \theta_1+ \kappa(-\Delta)^{\beta}\theta_1 = f, \qquad\qquad\qquad x \in \Omega, \hspace{2pt} t > 0,
\end{cases}\end{equation}
and
\begin{equation}\begin{cases}\label{eq1.2}
\partial_t\boldsymbol u_2 + \boldsymbol u_2\cdot \nabla \boldsymbol u_2 + \nu(-\Delta)^\alpha \boldsymbol u_2=- \nabla \pi_2 + \theta_2 \boldsymbol e_2, \qquad x \in \Omega, \hspace{2pt} t > 0, \\
\nabla \cdot \boldsymbol u_2=0,  \qquad\qquad\qquad\qquad\qquad\qquad\qquad x \in \Omega, \hspace{2pt} t > 0,\\
\partial_t \theta_2 +\boldsymbol u_2\cdot \nabla \theta_2+ \kappa(-\Delta)^{\beta}\theta_2 = g, \qquad\qquad\qquad x \in \Omega, \hspace{2pt} t > 0.
\end{cases}\end{equation}
We recall the Galerkin projections $P_m$ associated with the first $m$ Fourier modes of the operator $\Lambda$ and expand each solution in the form:
\[
\theta_1(x,t) = \sum_{k=1}^{\infty} \hat{\theta}_{k}^1(t)\omega_k(x),\qquad
\boldsymbol u_1(x,t) = \sum_{k=1}^{\infty}  \hat{\boldsymbol u}_{k}^1(t)\omega_k(x),
\]
and
\[
\theta_2(x,t) = \sum_{k=1}^{\infty} \hat{\theta}_{k}^2(t) \omega_k(x), \qquad
\boldsymbol u_2(x,t) = \sum_{k=1}^{\infty} \hat{\boldsymbol u}_{k}^2(t) \omega_k(x).
\]
The Galerkin projections correspond to the first $m$ Fourier modes:
\[
P_m \theta_1(x,t) = \sum_{k=1}^{m} \hat{\theta}_{k}^1(t) \omega_k(x), \qquad
P_m\boldsymbol u_1(x,t) = \sum_{k=1}^{m} \hat{\boldsymbol u}_{k}^1(t) \omega_k(x) ,
\]
and
\[
P_m\theta_2(x,t) = \sum_{k=1}^{m} \hat{\theta}_{k}^2(t) \omega_k(x), \qquad
P_m\boldsymbol u_2(x,t) = \sum_{k=1}^{m} \hat{\boldsymbol u}_{k}^2(t) \omega_k(x).
\]

Now we give the definition of determining modes for trajectories on the global attractor. See \cite{FMRT01} and \cite{CJRT12}.
\begin{definition}
The first m modes associated with $P_m$ are called the determining modes on the global attractor $\mathcal{A}$ if for two trajectories $(\theta_1(x,t), \boldsymbol u_1(x,t))$ and $(\theta_2(x,t), \boldsymbol u_2(x,t))$ on the global attractor $\mathcal{A}$, the condition 
\begin{equation} \label{2.0.5}
P_m (\theta_1(x,t), \boldsymbol u_1(x,t)) = P_m (\theta_2(x,t), \boldsymbol u_2(x,t)),  \qquad \forall t \in \mathbb{R}
\end{equation}
implies
\begin{equation}
(\theta_1(t), \boldsymbol u_1(t)) = (\theta_2(t), \boldsymbol u_2(t)),   \qquad \forall t \in \mathbb{R}.
\end{equation}
\end{definition}

\subsection{$H^{2\beta} \times H^{2\alpha}$-estimates for $(\theta, \boldsymbol u)$} \label{sec4.2} From \cite[Proposition 3.4]{HH15}, we could readily see that $\aiminnorm{\Lambda^{2\beta} \theta}$ and $\aiminnorm{\Lambda^{2\alpha} \theta}$ are uniformly bounded, we would like to find out the explicit dependence on the viscosity $\nu$ and the diffusivity $\kappa$ for the bounds of $\aiminnorm{\Lambda^{2\beta} \theta}$ and $\aiminnorm{\Lambda^{2\alpha} \boldsymbol u}$. Hence throughout this section, we emphasize that the constant C below is independent of the viscosity $\nu$ and the diffusivity $\kappa$.
\begin{proposition} \label{prop1}
Under the assumption of Theorem~\ref{thm2.1}, suppose $(\theta, \boldsymbol u)$ is on the global attractor $\mathcal{A}$. If $f \in L^{\frac{4}{2\beta-1}} \cap H^{\beta}$, then 
\begin{equation} \label{eqnn1}
\aiminnorm{\Lambda^{2\beta} \theta}^2+\aiminnorm{\Lambda^{2\alpha}\boldsymbol u}^2
\leq  C\left(\frac{(1+\kappa)^2e^{2\nu}}{\kappa^3\nu^2}\right)\aiminnorm{\Lambda^{\beta}f}^2e^{2M}.
\end{equation}
where $M=M(\alpha, \beta, \kappa, \nu, \aiminnorm{f}, \aiminnorm{f}_{L^{\frac{4}{2\beta-1}}})$ and the constant C is independent of $\nu$ and $\kappa$.
\end{proposition}
\begin{proof}
Let us first consider that $\beta \leq \alpha$.
Taking the inner product of  the equation \eqref{eq1.1.1}$_3$ with $\Lambda^{4\beta}\theta$ in $L^2$, we have
\begin{equation} \nonumber
\frac{1}{2} \frac{\mathrm{d}}{\mathrm{d}t} \aiminnorm{\Lambda^{2\beta} \theta}^2 + \aimininner{\boldsymbol u \cdot \nabla \theta}{\Lambda^{4\beta}\theta}+ \kappa \aiminnorm{\Lambda^{3\beta} \theta}^2 = \aimininner{\Lambda^{\beta}f}{\Lambda^{3\beta}\theta} \leq \frac{C}{\kappa}\aiminnorm{\Lambda^{\beta}f}^2+
\frac{\kappa}{4}\aiminnorm{\Lambda^{3\beta}\theta}^2.
\end{equation}
Let $\beta_1= 1/4+\beta/2$, such that $1/2 < \beta_1 < \beta$.
Since $\boldsymbol u$ is divergence free, applying Lemma~\ref{lem2.4} and choosing $p_1=q_1=1/(1-\beta_1)$, $p_2=q_2= 2/(2\beta_1-1)$, such that $1/p_1+1/p_2=1/q_1+1/q_2=1/2$, we have
\begin{equation} \label{eq4.2.1}
\begin{split}
|\aimininner{\boldsymbol u \cdot \nabla \theta}{\Lambda^{4\beta}\theta}| 
& \leq \aiminnorm{\Lambda^{2\beta-\beta_1}(\boldsymbol u \cdot \nabla \theta)}\aiminnorm{\Lambda^{2\beta+\beta_1}\theta} \\
& \leq C\aiminnorm{\Lambda^{1+2\beta-\beta_1}(\boldsymbol u \cdot \theta)} \aiminnorm{\Lambda^{2\beta+\beta_1}\theta} \\
& \leq C(\aiminnorm{\Lambda^{1+2\beta-\beta_1}\boldsymbol u}_{L^{p_1}}\aiminnorm{\theta}_{L^{p_2}}+ \aiminnorm{\Lambda^{1+2\beta-\beta_1}\theta}_{L^{q_1}}\aiminnorm{\boldsymbol u}_{L^{q_2}})\aiminnorm{\Lambda^{2\beta+\beta_1}\theta} \\
& \leq C(\aiminnorm{\Lambda^{2\beta+\beta_1}\theta}\aiminnorm{\Lambda^{2\beta+\beta_1} \boldsymbol u}\aiminnorm{\theta}_{L^{\frac{4}{2\beta-1}}} +
\aiminnorm{\Lambda^{2\beta+\beta_1}\theta}^2\aiminnorm{\boldsymbol u}_{L^{\frac{4}{2\beta-1}}}) \\
& \leq C(\aiminnorm{\Lambda^{2\beta+\beta_1}\theta}^2+\aiminnorm{\Lambda^{2\beta+\beta_1} \boldsymbol u}^2)\aiminnorm{\theta}_{L^{\frac{4}{2\beta-1}}} +
C\aiminnorm{\Lambda^{2\beta+\beta_1}\theta}^2\aiminnorm{\boldsymbol u}_{L^{\frac{4}{2\beta-1}}}.
\end{split}
\end{equation}
Since $(\theta(t), \boldsymbol u(t))$ is on the attractor $\mathcal{A}$ for all $t \in \mathbb{R}$, then we recall the uniform $L^{p}$-estimates for $(\theta, \boldsymbol u)$ in \cite[Section 3.1, 3.3]{HH15} that for all $p \geq 2$, $ \theta, \boldsymbol u \in L^{\infty}(0, \infty; L^p)$ with
\begin{equation} \label{eq4.2.0}
\aiminnorm{\theta(t)}_{L^p} \leq C\frac{\aiminnorm{f}_{L^p}}{\kappa},
\qquad \mathrm{and} \qquad
\aiminnorm{\boldsymbol u(t)}_{L^p} \leq 
C\aiminnorm{\boldsymbol u(t)}_{H^1} \leq
C\frac{e^{\nu}(1+\kappa)}{\nu^3\kappa^3} \aiminnorm{f}^2.
\end{equation}
where the constant $C=C(p)$ only depends on the exponent $p$.
Hence, setting
\[
A= \frac{\aiminnorm{f}_{L^{\frac{4}{2\beta-1}}}}{\kappa},  \qquad  \mathrm{and} \qquad
B=\frac{e^{\nu}(1+\kappa)}{\nu^3\kappa^3} \aiminnorm{f}^2,
\]
and using \eqref{eq4.2.0}, the interpolation inequality,  Young's inequality, and the assumption $\beta \leq \alpha$, we find from \eqref{eq4.2.1} 
that 
\begin{equation} \nonumber
\begin{split}
|\aimininner{\boldsymbol u \cdot \nabla \theta}{\Lambda^{4\beta}\theta}| 
& \leq  CA\aiminnorm{\Lambda^{2\beta+\beta_1}\theta}^2+
CA\aiminnorm{\Lambda^{2\beta+\beta_1} \boldsymbol u}^2
+CB\aiminnorm{\Lambda^{2\beta+\beta_1}\theta}^2 \\
& \leq  C(A+B)\aiminnorm{\Lambda^{2\beta} \theta}^{2-2\frac{\beta_1}{\beta}}\aiminnorm{\Lambda^{3\beta} \theta}^{2\frac{\beta_1}{\beta}}+ CA\aiminnorm{\Lambda^{2\beta} \boldsymbol u}^{2-2\frac{\beta_1}{\beta}}\aiminnorm{\Lambda^{3\beta} \boldsymbol u}^{2\frac{\beta_1}{\beta}} \\
& \leq  
C\kappa^{-\frac{2\beta+1}{2\beta-1}}(A+B)^{\frac{4\beta}{2\beta-1}}\aiminnorm{\Lambda^{2\beta} \theta}^2+\frac{\kappa}{4}\aiminnorm{\Lambda^{3\beta} \theta}^2  \\
& \hspace{10pt} + C\nu^{-\frac{2\beta+1}{2\beta-1}}A^{\frac{4\beta}{2\beta-1}}\aiminnorm{\Lambda^{2\alpha} \boldsymbol u}^2+
\frac{\nu}{4} \aiminnorm{\Lambda^{3\alpha} \boldsymbol u}^2.
\end{split}
\end{equation}
Therefore, we obtain
\begin{equation} \label{eq4.4.1}
\begin{split}
\frac{1}{2}\frac{\mathrm{d}}{\mathrm{d}t} \aiminnorm{\Lambda^{2\beta} \theta}^2 + \kappa \aiminnorm{\Lambda^{3\beta} \theta}^2 
& \leq  
C\kappa^{-\frac{2\beta+1}{2\beta-1}}(A+B)^{\frac{4\beta}{2\beta-1}}\aiminnorm{\Lambda^{2\beta} \theta}^2 +
C\nu^{-\frac{2\beta+1}{2\beta-1}}A^{\frac{4\beta}{2\beta-1}}\aiminnorm{\Lambda^{2\alpha} \boldsymbol u}^2 \\
& \hspace{10pt} + \frac{\nu}{4} \aiminnorm{\Lambda^{3\alpha} \boldsymbol u}^2  +  \frac{C}{\kappa} \aiminnorm{\Lambda^{\beta} f}^2.
\end{split}
\end{equation}

Next, we take the inner product of  the equation \eqref{eq1.1.1}$_1$ with $\Lambda^{4\alpha} \boldsymbol u$ in $L^2$. Since $\boldsymbol u$ is divergence free, we hence have
\[
\frac{1}{2} \frac{\mathrm{d}}{\mathrm{d}t} \aiminnorm{\Lambda^{2\alpha}\boldsymbol u}^2 +  \nu \aiminnorm{\Lambda^{3\alpha} \boldsymbol u}^2 = \aimininner{\theta \boldsymbol e_2}{\Lambda^{4\alpha}\boldsymbol u}
-\aimininner{\boldsymbol u \cdot \nabla \boldsymbol u}{\Lambda^{4\alpha}\boldsymbol u}.
\]
Similar to the estimate~\eqref{eq4.2.1}, we choose $\alpha_1 = 1/4+\alpha/2$, such that $1/2 < \alpha_1 < \alpha$. Then applying the Interpolation inequality, we find
\begin{equation} \nonumber
\begin{split}
|\aimininner{\boldsymbol u \cdot \nabla \boldsymbol u}{\Lambda^{4\alpha}\boldsymbol u}| 
&\leq C\aiminnorm{\Lambda^{2\alpha+\alpha_1} \boldsymbol u}^2
\aiminnorm{\boldsymbol u}_{L^{\frac{4}{2\alpha-1}}}  \\
& \leq CB\aiminnorm{\Lambda^{2\alpha+\alpha_1} \boldsymbol u}^2 \\
& \leq C\nu^{-\frac{2\alpha+1}{2\alpha-1}}B^{\frac{4\alpha}
{2\alpha-1}}\aiminnorm{\Lambda^{2\alpha} \boldsymbol u}^2
+ \frac{\nu}{4}\aiminnorm{\Lambda^{3\alpha} \boldsymbol u}^2.
\end{split}
\end{equation}
In addition, applying Poincar\'e's and the Cauchy-Schwarz inequalities, we have
\[
\begin{split}
|\aimininner{\theta \boldsymbol e_2}{\Lambda^{4\alpha}\boldsymbol u}| & = |\aimininner{\Lambda^{\alpha} \theta \boldsymbol e_2}{\Lambda^{3\alpha}\boldsymbol u}|
\leq \aiminnorm{\Lambda^{\alpha} \theta}
\aiminnorm{\Lambda^{3\alpha}\boldsymbol u}
\leq \aiminnorm{\Lambda^{2\beta} \theta}
\aiminnorm{\Lambda^{3\alpha}\boldsymbol u} \\
& \leq \frac{1}{\nu}\aiminnorm{\Lambda^{2\beta} \theta}^2+
\frac{\nu}{4}\aiminnorm{\Lambda^{3\alpha}\boldsymbol u}^2.
\end{split}
\]
Thus, we obtain 
\begin{equation} \label{eq4.5.1}
\frac{\mathrm{d}}{\mathrm{d}t} \aiminnorm{\Lambda^{2\alpha}\boldsymbol u}^2 +  \nu \aiminnorm{\Lambda^{3\alpha} \boldsymbol u}^2 \leq C\nu^{-\frac{2\alpha+1}{2\alpha-1}}B^{\frac{4\alpha}
{2\alpha-1}}\aiminnorm{\Lambda^{2\alpha} \boldsymbol u}^2+
\frac{1}{\nu}\aiminnorm{\Lambda^{2\beta} \theta}^2.
\end{equation}
Summing equations \eqref{eq4.4.1} and \eqref{eq4.5.1}, we have,
\begin{equation} \label{eq4.6.1}
\begin{split}
& \frac{\mathrm{d}}{\mathrm{d}t} \big(\aiminnorm{\Lambda^{2\beta} \theta}^2+\aiminnorm{\Lambda^{2\alpha} \boldsymbol u}^2\big) + \kappa \aiminnorm{\Lambda^{3\beta} \theta}^2  + \nu \aiminnorm{\Lambda^{3\alpha} \boldsymbol u}^2  \\
\leq  & \hspace{3pt}
C\left(\kappa^{-\frac{2\beta+1}{2\beta-1}}(A+B)^{\frac{4\beta}{2\beta-1}}+\frac{1}{\nu}\right)\aiminnorm{\Lambda^{2\beta}\theta}^2 \\ & + C\left(\nu^{-\frac{2\beta+1}{2\beta-1}}A^{\frac{4\beta}{2\beta-1}}+ \nu^{-\frac{2\alpha+1}{2\alpha-1}}B^{\frac{4\alpha}
{2\alpha-1}}\right)
\aiminnorm{\Lambda^{2\alpha} \boldsymbol u}^2+
\frac{C}{\kappa}\aiminnorm{\Lambda^{\beta} f}^2.
\end{split}
\end{equation}
Let
\begin{equation} \label{e0}
M_1 := \mathrm{max} \{ \kappa^{-\frac{2\beta+1}{2\beta-1}}(A+B)^{\frac{4\beta}{2\beta-1}}+\frac{1}{\nu}, \hspace{5pt} \nu^{-\frac{2\beta+1}{2\beta-1}}A^{\frac{4\beta}{2\beta-1}}+ \nu^{-\frac{2\alpha+1}{2\alpha-1}}B^{\frac{4\alpha}
{2\alpha-1}} \}, 
\end{equation}
then it follows from \eqref{eq4.6.1} that
\begin{equation} \label{e4.1.1}
\frac{\mathrm{d}}{\mathrm{d}t} \big(\aiminnorm{\Lambda^{2\beta} \theta}^2+\aiminnorm{\Lambda^{2\alpha} \boldsymbol u}^2\big) + \kappa \aiminnorm{\Lambda^{3\beta} \theta}^2  + \nu \aiminnorm{\Lambda^{3\alpha} \boldsymbol u}^2 
\leq CM_1\big(\aiminnorm{\Lambda^{2\beta} \theta}^2+\aiminnorm{\Lambda^{2\alpha} \boldsymbol u}^2\big)+ \frac{C}{\kappa}\aiminnorm{\Lambda^{\beta}f}^2.
\end{equation}

In order to apply the Uniform Gronwall inequality and obtain the uniform bounds for $\Lambda^{2\beta} \theta$ and $\Lambda^{2\alpha} \boldsymbol u$, we have to find the uniform time average bounds for $\Lambda^{2\beta} \theta$ and $\Lambda^{2\alpha} \boldsymbol u$.
Taking the inner product of \eqref{eq1.1.1}
with $(\Lambda^{\beta} \theta, \Lambda^{\alpha} \boldsymbol u)$ in $L^2$ and using analogous arguments as for \eqref{e4.1.1}, we have
\begin{equation} \label{eqn4.4.1}
\frac{\mathrm{d}}{\mathrm{d}t} \big(\aiminnorm{\Lambda^{\beta} \theta}^2+\aiminnorm{\Lambda^{\alpha} \boldsymbol u}^2\big) + \kappa \aiminnorm{\Lambda^{2\beta} \theta}^2  + \nu \aiminnorm{\Lambda^{2\alpha} \boldsymbol u}^2 
\leq CM_1\big(\aiminnorm{\Lambda^{\beta} \theta}^2+\aiminnorm{\Lambda^{\alpha} \boldsymbol u}^2\big)+ \frac{C}{\kappa}\aiminnorm{f}^2.
\end{equation}
It has been shown in \cite[Section 3.1, 3.3]{HH15} that the time averages of $\aiminnorm{\Lambda^{\beta} \theta}^2$ and $\aiminnorm{\Lambda^{\alpha} \boldsymbol u}^2$ are uniformly bounded. That is, for $t \geq t_1(\theta_0, \boldsymbol u_0)$ large enough,
\[
\int^{t+1}_t \aiminnorm{\Lambda^{\beta} \theta}^2 {\mathrm{d}s} 
\leq C\frac{1+\kappa}{\kappa^3} \aiminnorm{f}^2,
\]
and 
\[
\int^{t+1}_t \aiminnorm{\Lambda^{\alpha} \boldsymbol u}^2 {\mathrm{d}s} 
\leq \frac{1}{\nu}\aiminnorm{\boldsymbol u}^2+\frac{1}{\nu^2} \int^{t+1}_t \aiminnorm{\theta}^2 {\mathrm{d}s} 
\leq C\frac{1+\nu}{\kappa^3\nu^2} \aiminnorm{f}^2.
\]
Applying Uniform Gronwall Lemma~\ref{lemma 2.1} on the differential inequality \eqref{eqn4.4.1} with $a_1=CM_1$, $a_2=\frac{C}{\kappa} \aiminnorm{f}^2$ and $a_3=C\frac{(1+\kappa)(1+\nu)^2}{\kappa^3\nu^2} \aiminnorm{f}^2$,  we have for $t \geq t_2(\theta_0, \boldsymbol u_0)$,
\begin{equation} \nonumber
\begin{split}
\aiminnorm{\Lambda^{\beta} \theta}^2+\aiminnorm{\Lambda^{\alpha}\boldsymbol u}^2 
& \leq C\left(\frac{(1+\kappa)(1+\nu)^2}{\kappa^3\nu^2} +\frac{1}{\kappa}\right)\aiminnorm{f}^2e^{M_1} \\
& \leq C\left(\frac{(1+\kappa)^2(1+\nu)^2}{\kappa^3\nu^2}\right)\aiminnorm{f}^2e^{M_1}.
\end{split}
\end{equation}
In additional, for $t \geq t_2(\theta_0, \boldsymbol u_0)$,
\begin{equation} \nonumber
\begin{split}
\int_{t}^{t+1} \aiminnorm{\Lambda^{2\beta} \theta}^2+\aiminnorm{\Lambda^{2\alpha}\boldsymbol u}^2 {\mathrm{d}s}
&\leq C\left(\frac{(1+\kappa)^2(1+\nu)^2}{\kappa^3\nu^2}\right)\aiminnorm{f}^2e^{M_1}+ \frac{C}{\kappa}\aiminnorm{f}^2 \\
& \leq C\left(\frac{(1+\kappa)^2(1+\nu)^2}{\kappa^3\nu^2}\right)\aiminnorm{f}^2e^{M_1}.
\end{split}
\end{equation}
Applying Uniform Gronwall Lemma~\ref{lemma 2.1} again on the differential inequality \eqref{e4.1.1} with $a_1=CM_1$, $a_2=\frac{C}{\kappa} \aiminnorm{\Lambda^{\beta} f}^2$ and $a_3=C\frac{(1+\kappa)^2(1+\nu)^2}{\kappa^3\nu^2} \aiminnorm{f}^2e^M_1$,  we have for $t \geq t_3(\theta_0, \boldsymbol u_0)$,
\begin{equation} \nonumber
\begin{split}
\aiminnorm{\Lambda^{2\beta} \theta}^2+\aiminnorm{\Lambda^{2\alpha} \boldsymbol u}^2 
&\leq C\left(\frac{1}{\kappa} \aiminnorm{\Lambda^{\beta} f}^2+\frac{(1+\kappa)^2(1+\nu)^2}{\kappa^3\nu^2} \aiminnorm{f}^2e^{M_1}\right)e^{M_1} \\
& \leq C\left(\frac{(1+\kappa)^2(1+\nu)^2}{\kappa^3\nu^2}\right)\aiminnorm{\Lambda^{\beta}f}^2e^{2M_1}.
\end{split}
\end{equation}
Since we assume that the solution $(\theta, \boldsymbol u)$ is on the global attractor, we can shift the initial time, so that  
\begin{equation} \label{eq4.7}
\aiminnorm{\Lambda^{2\beta} \theta}^2+\aiminnorm{\Lambda^{2\alpha} \boldsymbol u}^2 
\leq C\left(\frac{(1+\kappa)^2(1+\nu)^2}{\kappa^3\nu^2}\right)\aiminnorm{\Lambda^{\beta}f}^2e^{2M_1}, \qquad \forall t \in \mathbb{R}.
\end{equation}
For the case $\beta > \alpha$, similar to the equation~\eqref{e4.1.1}, we find
\begin{equation} \label{w1}
\begin{split}
&\frac{\mathrm{d}}{\mathrm{d}t} \big(\aiminnorm{\Lambda^{2\beta} \theta}^2+\aiminnorm{\Lambda^{2\beta+\alpha} \boldsymbol u}^2\big) + \kappa \aiminnorm{\Lambda^{3\beta} \theta}^2  + \nu \aiminnorm{\Lambda^{2\beta+2\alpha} \boldsymbol u}^2 \\ 
\leq & \hspace{2pt} CM_2\big(\aiminnorm{\Lambda^{2\beta} \theta}^2+\aiminnorm{\Lambda^{2\beta+\alpha} \boldsymbol u}^2\big)+ \frac{C}{\kappa}\aiminnorm{\Lambda^{\beta}f}^2,
\end{split}
\end{equation}
where
\begin{equation} \label{e123}
M_2 := \mathrm{max} \{ \kappa^{-\frac{\alpha+\beta}{\alpha+\beta-1}}(A_1+B)^{\frac{2\alpha+2\beta}{\alpha+\beta-1}}+\frac{1}{\nu}, \hspace{5pt} \nu^{-\frac{\beta-\alpha}{3\alpha-\beta}}A_1^{\frac{2\alpha}{3\alpha-\beta}}+ \nu^{-\frac{2\alpha+1}{2\alpha-1}}B^{\frac{4\alpha}{2\alpha-1}} \}, 
\end{equation}
and $A_1=\frac{\aiminnorm{f}_{L^{\frac{2}{\alpha+\beta-1}}}}{\kappa}$. In addition, we have
\begin{equation} \label{w2}
\begin{split}
&\frac{\mathrm{d}}{\mathrm{d}t} \big(\aiminnorm{\Lambda^{\beta} \theta}^2+\aiminnorm{\Lambda^{\beta+\frac{\alpha}{2}} \boldsymbol u}^2\big) + \kappa \aiminnorm{\Lambda^{2\beta} \theta}^2  + \nu \aiminnorm{\Lambda^{2\beta+\alpha} \boldsymbol u}^2 \\
\leq & \hspace{2pt} CM_2\big(\aiminnorm{\Lambda^{\beta} \theta}^2+\aiminnorm{\Lambda^{\beta+\frac{\alpha}{2}} \boldsymbol u}^2\big)+ \frac{C}{\kappa}\aiminnorm{f}^2.
\end{split}
\end{equation}
We recall the results from \cite[Section 3.1, 3.3]{HH15} that the time averages of $\aiminnorm{\Lambda^{\beta} \theta}^2$ and $\aiminnorm{\Lambda^{\beta+\frac{\alpha}{2}} \boldsymbol u}^2$ are uniformly bounded. That is, for $t \geq t_4(\theta_0, \boldsymbol u_0)$ large enough,
\[
\int^{t+1}_t \aiminnorm{\Lambda^{\beta} \theta}^2 {\mathrm{d}s} 
\leq C\frac{1+\kappa}{\kappa^3} \aiminnorm{f}^2,
\]
and 
\[
\int^{t+1}_t \aiminnorm{\Lambda^{\beta+\frac{\alpha}{2}} \boldsymbol u}^2 {\mathrm{d}s} 
\leq \int^{t+1}_t \aiminnorm{\Lambda^{1+\alpha} \boldsymbol u}^2 {\mathrm{d}s} 
\leq C\frac{e^{2\nu}(1+\kappa)}{\kappa^3\nu^2} \aiminnorm{f}^2.
\]
Applying Uniform Gronwall Lemma~\ref{lemma 2.1} on the differential inequality \eqref{w2} with $a_1=CM_2$, $a_2=\frac{C}{\kappa} \aiminnorm{f}^2$ and $a_3=C\frac{(1+\kappa)e^{2\nu}}{\kappa^3\nu^2} \aiminnorm{f}^2$,  we have for $t \geq t_5(\theta_0, \boldsymbol u_0)$,
\begin{equation} \nonumber
\begin{split}
\aiminnorm{\Lambda^{\beta} \theta}^2+\aiminnorm{\Lambda^{\beta+\frac{\alpha}{2}}\boldsymbol u}^2 
& \leq C\left(\frac{(1+\kappa)e^{2\nu}}{\kappa^3\nu^2} +\frac{1}{\kappa}\right)\aiminnorm{f}^2e^{M_2}\\
& \leq C\left(\frac{(1+\kappa)^2e^{2\nu}}{\kappa^3\nu^2}\right)\aiminnorm{f}^2e^{M_2}.
\end{split}
\end{equation}
Applying Uniform Gronwall Lemma~\ref{lemma 2.1} again on the differential inequality \eqref{w1} with $a_1=CM_1$, $a_2=\frac{C}{\kappa} \aiminnorm{\Lambda^{\beta} f}^2$ and $a_3=C\frac{(1+\kappa)^2e^{2\nu}}{\kappa^3\nu^2} \aiminnorm{f}^2e^{M_1}$,  we have for $t \geq t_6(\theta_0, \boldsymbol u_0)$,
\begin{equation} \nonumber
\begin{split}
\aiminnorm{\Lambda^{2\beta} \theta}^2+\aiminnorm{\Lambda^{2\beta+\alpha} \boldsymbol u}^2 
&\leq C\left(\frac{1}{\kappa} \aiminnorm{\Lambda^{\beta} f}^2+\frac{(1+\kappa)^2(1+\nu)^2}{\kappa^3\nu^2} \aiminnorm{f}^2e^{M_2}\right)e^{M_2} \\
& \leq C\left(\frac{(1+\kappa)^2e^{2\nu}}{\kappa^3\nu^2}\right)\aiminnorm{\Lambda^{\beta}f}^2e^{2M_2}.
\end{split}
\end{equation}
Since we assume that the solution $(\theta, \boldsymbol u)$ is on the global attractor, we can shift the initial time, so that  
\begin{equation} \label{w7}
\aiminnorm{\Lambda^{2\beta} \theta}^2+\aiminnorm{\Lambda^{2\beta+\alpha} \boldsymbol u}^2 
\leq C\left(\frac{(1+\kappa)^2e^{2\nu}}{\kappa^3\nu^2}\right)\aiminnorm{\Lambda^{\beta}f}^2e^{2M_1}, \qquad \forall t \in \mathbb{R}.
\end{equation}
Since $\alpha, \beta \in (1/2, 1)$, we have 
\[
H^{2\alpha} \subset \subset H^{2\beta+\alpha} \qquad
\mathrm{and} \qquad L^{\frac{2}{\alpha+\beta-1}} \subset
L^{\frac{4}{2\beta-1}}.
\]
Let
\begin{equation} \label{e111} 
M := \mathrm{max} \{M_1, M_2 \},
\end{equation} 
we hence conclude \eqref{eqnn1}, based on equations \eqref{eq4.7}
and \eqref{w7}.
\end{proof}
\subsection{Main results} \label{sec4.3}
\begin{theorem} \label{thm1}
Under the assumptions of Theorem~\ref{thm2.1}, let $(\theta_1, \boldsymbol u_1)$ and $(\theta_2, \boldsymbol u_2)$ be two trajectories of the system \eqref{eq1.1.1} on the attractor $\mathcal{A}$. If $P_m(\theta_1(t), \boldsymbol u_1(t)) =
P_m(\theta_2(t), \boldsymbol u_2(t))$ for all $t \in \mathbb{R}$, and for some integer $m > 0$ large enough, such that
\begin{equation} \label{e1}
\lambda_{m+1}^{\alpha-\frac{1}{2}} \geq \frac{2C(\kappa N^{\frac{1}{2}}+N+1)}{\kappa \nu},
\end{equation}
where the number $N$ is defined in \eqref{N} below.  
Then,  we have $(\theta_1(t), \boldsymbol u_1(t)) = (\theta_2(t), \boldsymbol u_2(t))$, for all $t \in \mathbb{R}$.
\end{theorem}
\begin{proof}
Let $(\theta_1, \boldsymbol u_1)$, $(\theta_2, \boldsymbol u_2)$ be the solutions on the attractor $\mathcal{A}$ and $(\eta, \boldsymbol w) = (\theta_1-\theta_2, \boldsymbol u_1 - \boldsymbol u_2)$. Then $(\eta, \boldsymbol w)$ satisfies:
\begin{equation}\begin{cases}\label{eq2.1.1}
\partial_t\boldsymbol w + \boldsymbol u_1\cdot \nabla \boldsymbol w + \boldsymbol w \cdot \nabla \boldsymbol u_2 + \nu(-\Delta)^\alpha \boldsymbol w= -\nabla(\pi_1-\pi_2)+\eta \boldsymbol e_2, \\
\partial_t \eta +\boldsymbol w\cdot \nabla \theta_1+ \boldsymbol u_2 \cdot \nabla \eta + \kappa(-\Delta)^{\beta}\eta = 0.
\end{cases}
\end{equation}
We now take the inner product of the equation \eqref{eq2.1.1}$_1$, \eqref{eq2.1.1}$_2$ with $Q_m \boldsymbol w$, $Q_m \eta$ in $L^2$ respectively, where $Q_m = I - P_m$. Since $P_m(\eta, \boldsymbol w) = 0$ and by integration by parts, we obtain
\[
\aimininner{\boldsymbol u_1\cdot \nabla \boldsymbol w}{Q_m \boldsymbol w}= \aimininner{\boldsymbol u_1\cdot \nabla P_m\boldsymbol w}{Q_m \boldsymbol w}+\aimininner{\boldsymbol u_1\cdot \nabla Q_m\boldsymbol w}{Q_m \boldsymbol w}= 0,
\]
and 
\[
\aimininner{\boldsymbol u_2 \cdot \nabla \eta}{Q_m \eta}= \aimininner{\boldsymbol u_2 \cdot \nabla P_m\eta}{Q_m \eta}+ \aimininner{\boldsymbol u_2 \cdot \nabla Q_m\eta}{Q_m \eta}= 0.
\]
Then we find,
\begin{equation}\begin{cases}\label{eq5.2.2}
\frac{1}{2} \frac{\mathrm{d}}{\mathrm{d}t} \aiminnorm{Q_m\boldsymbol w}^2 + \nu \aiminnorm{\Lambda^{\alpha} Q_m \boldsymbol w}^2= - \aimininner{\boldsymbol w \cdot \nabla \boldsymbol u_2}{Q_m \boldsymbol w} + \aimininner{\eta \boldsymbol e_2}{Q_m \boldsymbol w}, \\
\frac{1}{2} \frac{\mathrm{d}}{\mathrm{d}t} \aiminnorm{Q_m \eta}^2 + \kappa \aiminnorm{\Lambda^{\beta} Q_m \eta}^2= -\aimininner{\boldsymbol w \cdot \nabla \theta_1}{Q_m \eta}.
\end{cases}\end{equation}

We now estimate the term $\aimininner{\boldsymbol w \cdot \nabla \boldsymbol u_2}{Q_m \boldsymbol w}$. Since $P_m \boldsymbol w = 0$, 
\begin{equation} \label{eq5.3}
\aimininner{\boldsymbol w \cdot \nabla \boldsymbol u_2}{Q_m \boldsymbol w} =
\aimininner{P_m\boldsymbol w \cdot \nabla \boldsymbol u_2}{Q_m \boldsymbol w}+
\aimininner{Q_m\boldsymbol w \cdot \nabla \boldsymbol u_2}{Q_m \boldsymbol w} = \aimininner{Q_m\boldsymbol w \cdot \nabla \boldsymbol u_2}{Q_m \boldsymbol w},
\end{equation}
and since $\boldsymbol u_2$ is divergence free, we have
\begin{equation} \label{eq5.4}
\begin{split}
|\aimininner{Q_m \boldsymbol w \cdot \nabla \boldsymbol u_2}{Q_m \boldsymbol w} | &=
|\aimininner{\Lambda^{-\alpha}(Q_m \boldsymbol w \cdot \nabla \boldsymbol u_2)}{\Lambda^{\alpha} Q_m \boldsymbol w} | \\
& \leq \aiminnorm{\Lambda^{-\alpha}(Q_m \boldsymbol w \cdot \nabla \boldsymbol u_2)}
\aiminnorm{\Lambda^{\alpha} Q_m \boldsymbol w} \\
& \leq C\aiminnorm{\Lambda^{1-\alpha}(Q_m \boldsymbol w \otimes\boldsymbol u_2)}
\aiminnorm{\Lambda^{\alpha} Q_m \boldsymbol w}.
\end{split}
\end{equation}
Let $p_1 = 4/(3-2\alpha)$, $p_2 = 4/(2\alpha-1)$, 
$q_1 = 4$, and $q_2 = 4$ such that $1/p_1+1/p_2=1/q_1+1/q_2=1/2$. Applying Lemma~\ref{lem2.4} and the Sobolev inequality in Lemma~\ref{lemma2.0}, we find
\begin{equation} \label{eq5.5}
\begin{split}
\aiminnorm{\Lambda^{1-\alpha}(Q_m \boldsymbol w \otimes\boldsymbol u_2)}
&\leq C(\aiminnorm{\Lambda^{1-\alpha}Q_m \boldsymbol w}_{L^{p_1}}
\aiminnorm{\boldsymbol u_2}_{L^{p_2}}+
\aiminnorm{\Lambda^{1-\alpha}\boldsymbol u_2}_{L^{q_1}}
\aiminnorm{Q_m \boldsymbol w}_{L^{q_2}}) \\
& \leq C\aiminnorm{\Lambda^{\frac{1}{2}}Q_m \boldsymbol w}
\aiminnorm{\Lambda^{\frac{3}{2}-\alpha} \boldsymbol u_2}.
\end{split}
\end{equation}
Applying Poincar\'e's inequality and using \eqref{eqnn1}, we find
\[
\aiminnorm{\Lambda^{\frac{3}{2}-\alpha}\boldsymbol u_2}^2
\leq \aiminnorm{\Lambda^{2\alpha} \boldsymbol u_2}^2
\leq C\left(\frac{(1+\kappa)^2e^{2\nu}}{\kappa^3\nu^2}\right)\aiminnorm{\Lambda^{\beta}f}^2e^{2M}, 
\]
and
\[
\aiminnorm{\Lambda^{\frac{1}{2}}Q_m \boldsymbol w}^2
\leq \lambda_{m+1}^{1-2\alpha} \aiminnorm{\Lambda^{\alpha}Q_m\boldsymbol w}^2.
\]
Hence, we deduce that
\begin{equation} \label{eq5.6}
\begin{split}
|\aimininner{Q_m \boldsymbol w \cdot \nabla \boldsymbol u_2}{Q_m \boldsymbol w} | 
&\leq C\aiminnorm{\Lambda^{\frac{1}{2}}Q_m \boldsymbol w}
\aiminnorm{\Lambda^{\frac{3}{2}-\alpha} \boldsymbol u_2}\aiminnorm{\Lambda^{\alpha} Q_m \boldsymbol w} \\
& \leq C\left(\frac{(1+\kappa)^2e^{2\nu}}{\kappa^3\nu^2}\aiminnorm{\Lambda^{\beta}f}^2e^{2M}\right)^{\frac{1}{2}} \lambda_{m+1}^{\frac{1}{2}-\alpha}\aiminnorm{\Lambda^{\alpha} Q_m \boldsymbol w}^2.
\end{split}
\end{equation}
Next, we estimate $\aimininner{\boldsymbol w \cdot \nabla \theta_1}{Q_m \eta}$. Since $P_m \boldsymbol w= 0$, we have
\[
\aimininner{\boldsymbol w \cdot \nabla \theta_1}{Q_m \eta} =
\aimininner{P_m\boldsymbol w \cdot \nabla \theta_1}{Q_m \eta}+
\aimininner{Q_m\boldsymbol w \cdot \nabla \theta_1}{Q_m \eta} =
\aimininner{Q_m\boldsymbol w \cdot \nabla \theta_1}{Q_m \eta}
\]
Similarly to \eqref{eq5.4}, since $\boldsymbol w$ is divergence free, we have
\begin{equation} \label{eq5.8}
|\aimininner{Q_m \boldsymbol w \cdot \nabla \theta_1}{Q_m \eta} | 
\leq \aiminnorm{\Lambda^{1-\beta}(Q_m \boldsymbol w \cdot \theta_1)}
\aiminnorm{\Lambda^{\beta} Q_m \eta}.
\end{equation}
Let $p_1 = 4/(3-2\beta)$, $p_2 = 4/(2\beta-1)$, 
$q_1 = 4$ and $q_2 = 4$ such that $1/p_1+1/p_2=1/q_1+1/q_2=1/2$, we have 
\begin{equation} \label{eq5.9}
\begin{split}
\aiminnorm{\Lambda^{1-\beta}(Q_m \boldsymbol w \cdot \theta_1)}
&\leq  C(\aiminnorm{\Lambda^{1-\beta}Q_m \boldsymbol w}_{L^{p_1}}
\aiminnorm{\theta_1}_{L^{p_2}}+
\aiminnorm{\Lambda^{1-\beta}\theta_1}_{L^{q_1}}
\aiminnorm{Q_m \boldsymbol w}_{L^{q_2}}) \\
& \leq C\aiminnorm{\Lambda^{\frac{1}{2}}Q_m \boldsymbol w}
\aiminnorm{\Lambda^{\frac{3}{2}-\beta} \theta_1}.
\end{split} 
\end{equation}
Applying Poincar\'e's inequality and \eqref{eqnn1}, we find
\[
\aiminnorm{\Lambda^{\frac{3}{2}-\beta}\theta_1}^2 \leq
\aiminnorm{\Lambda^{2\beta}\theta_1}^2
\leq C\left(\frac{(1+\kappa)^2e^{2\nu}}{\kappa^3\nu^2}\right)\aiminnorm{\Lambda^{\beta}f}^2e^{2M}.
\]
Hence, by the Cauchy-Schwarz inequality, we deduce
\begin{equation} \label{eq5.10}
\begin{split}
|\aimininner{Q_m \boldsymbol w \cdot \nabla \theta_1}{Q_m \eta} | 
&\leq C\aiminnorm{\Lambda^{\frac{1}{2}}Q_m \boldsymbol w}
\aiminnorm{\Lambda^{\frac{3}{2}-\beta} \theta_1}\aiminnorm{\Lambda^{\beta} Q_m \eta} \\
&\leq \frac{C}{\kappa}\aiminnorm{\Lambda^{\frac{1}{2}}Q_m \boldsymbol w}^2
\aiminnorm{\Lambda^{\frac{3}{2}-\beta} \theta_1}^2+ \frac{\kappa}{4}\aiminnorm{\Lambda^{\beta} Q_m \eta}^2 \\
& \leq C\left(\frac{(1+\kappa)^2e^{2\nu}}{\kappa^4\nu^2}\right)\aiminnorm{\Lambda^{\beta}f}^2e^{2M} 
\lambda_{m+1}^{1-2\alpha}\aiminnorm{\Lambda^{\alpha} Q_m \boldsymbol w}^2+
\frac{\kappa}{4}\aiminnorm{\Lambda^{\beta} Q_m \eta}^2.
\end{split}
\end{equation}
Finally, we estimate the term $\aimininner{\eta \boldsymbol e_2}{Q_m \boldsymbol w}$. Since $P_m \eta = 0$,
then 
\[
\aimininner{\eta \boldsymbol e_2}{Q_m \boldsymbol w} = \aimininner{P_m\eta \boldsymbol e_2}{Q_m \boldsymbol w}+\aimininner{Q_m\eta \boldsymbol e_2}{Q_m \boldsymbol w} =\aimininner{Q_m\eta \boldsymbol e_2}{Q_m \boldsymbol w},
\] 
and
\begin{equation} \nonumber
\begin{split}
|\aimininner{Q_m\eta \boldsymbol e_2}{Q_m \boldsymbol w}| 
& = |\aimininner{\Lambda^{\beta}Q_m\eta \boldsymbol e_2}{\Lambda^{-\beta}Q_m \boldsymbol w}|  
\leq \aiminnorm{\Lambda^{\beta}Q_m\eta}\aiminnorm{\Lambda^{-\beta}Q_m\boldsymbol w}  \\
&\leq \frac{\kappa}{4}\aiminnorm{\Lambda^{\beta} Q_m\eta}^2+ \frac{1}{\kappa}\aiminnorm{\Lambda^{-\beta}Q_m\boldsymbol w}^2 \\
&\leq \frac{\kappa}{4}\aiminnorm{\Lambda^{\beta} Q_m\eta}^2+ \frac{1}{\kappa}\lambda_{m+1}^{-2(\alpha+\beta)}\aiminnorm{\Lambda^{\alpha} Q_m\boldsymbol w}^2.
\end{split}
\end{equation}
Let us denote 
\begin{equation} \label{N}
N:= \frac{(1+\kappa)^2e^{2\nu}}{\kappa^3\nu^2}\aiminnorm{\Lambda^{\beta}f}^2e^{2M}.
\end{equation}
Therefore, we arrive the differential inequality
\begin{equation}\begin{cases}\label{eq5.2.2}
\frac{1}{2} \frac{\mathrm{d}}{\mathrm{d}t} \aiminnorm{Q_m\boldsymbol w}^2 + \nu \aiminnorm{\Lambda^{\alpha} Q_m \boldsymbol w}^2 \leq (CN^{\frac{1}{2}} \lambda_{m+1}^{\frac{1}{2}-\alpha}+\frac{1}{\kappa}\lambda_{m+1}^{-2(\alpha+\beta)})\aiminnorm{\Lambda^{\alpha} Q_m \boldsymbol w}^2+\frac{\kappa}{4}\aiminnorm{\Lambda^{\beta} Q_m\eta}^2, \\
\frac{1}{2} \frac{\mathrm{d}}{\mathrm{d}t} \aiminnorm{Q_m \eta}^2 + \kappa \aiminnorm{\Lambda^{\beta} Q_m \eta}^2 \leq
CN\frac{1}{\kappa} \lambda_{m+1}^{1-2\alpha} \aiminnorm{\Lambda^{\alpha} Q_m \boldsymbol w}^2+
\frac{\kappa}{4}\aiminnorm{\Lambda^{\beta} Q_m \eta}^2.
\end{cases}\end{equation}
Summing the above two differential inequalities, we obtain,
\begin{equation}\nonumber
\begin{split}
&\frac{1}{2}\frac{\mathrm{d}}{\mathrm{d}t} (\aiminnorm{Q_m\boldsymbol w}^2 
+ \aiminnorm{Q_m \eta}^2)+ \frac{\kappa}{2}\aiminnorm{\Lambda^{\beta} Q_m \eta}^2 \\
+ & (\nu-CN^{\frac{1}{2}} \lambda_{m+1}^{\frac{1}{2}-\alpha}-\frac{1}{\kappa}\lambda_{m+1}^{-2(\alpha+\beta)}-CN\frac{1}{\kappa} \lambda_{m+1}^{1-2\alpha} )\aiminnorm{\Lambda^{\alpha} Q_m \boldsymbol w}^2 \leq 0.
\end{split}
\end{equation}
Since $\lambda_{m+1} > 1$, then $\lambda_{m+1}^{-2(\alpha+\beta)} <\lambda_{m+1}^{1-2\alpha}$. Hence,
\begin{equation}\label{eq5.12.1}
\frac{1}{2}\frac{\mathrm{d}}{\mathrm{d}t} (\aiminnorm{Q_m\boldsymbol w}^2 
+ \aiminnorm{Q_m \eta}^2)+ \frac{\kappa}{2}\aiminnorm{\Lambda^{\beta} Q_m \eta}^2 
+ \left(\nu-\lambda_{m+1}^{\frac{1}{2}-\alpha}(CN^{\frac{1}{2}}+\frac{1}{\kappa}+CN\frac{1}{\kappa})\right)\aiminnorm{\Lambda^{\alpha} Q_m \boldsymbol w}^2 \leq 0.
\end{equation}
Under the conditions \eqref{e1}, \eqref{eq5.12.1} implies 
\begin{equation}\label{eq5.12}
\frac{\mathrm{d}}{\mathrm{d}t} (\aiminnorm{Q_m\boldsymbol w}^2 
+ \aiminnorm{Q_m \eta}^2) 
+ \nu\aiminnorm{\Lambda^{\alpha} Q_m \boldsymbol w}^2 
+ \kappa\aiminnorm{\Lambda^{\beta} Q_m \eta}^2 \leq 0.
\end{equation}
Hence,
\begin{equation}\label{eq5.13}
\frac{\mathrm{d}}{\mathrm{d}t} (\aiminnorm{Q_m\boldsymbol w}^2 
+ \aiminnorm{Q_m \eta}^2) 
+ \sigma(\aiminnorm{\Lambda^{\alpha} Q_m \boldsymbol w}^2 
+ \aiminnorm{\Lambda^{\beta} Q_m \eta}^2) \leq 0,
\end{equation}
where $\sigma = \mathrm{min} \{\kappa, \nu\}$. Now, integrating \eqref{eq5.13} from $t_0$ to $t$, we have
\begin{equation}\label{eq5.14}
\aiminnorm{Q_m\boldsymbol w(t)}^2 
+ \aiminnorm{Q_m \eta(t)}^2 
\leq (\aiminnorm{Q_m \boldsymbol w(t_0)}^2 
+ \aiminnorm{Q_m \eta(t_0)}^2)e^{r(t_0-t)}.
\end{equation}
Thus, it provides $Q_m\boldsymbol w(t) = Q_m\eta(t) = 0$ for all $t \in \mathbb{R}$, by taking $t_0 \rightarrow -\infty$. We thus finish the proof of Theorem~\ref{thm1}.
\end{proof} 
\begin{remark}
In \cite{FMRT01}, it was proved that for $m \rightarrow \infty$, we have $\lambda_m \sim c\lambda_1^{\frac{1}{2}} m^{\frac{1}{2}}$, where $c$ is a nondimensional constant. We can conclude that \eqref{e1} provides the number of determining modes $m$, such that
\[
\begin{split}
m &\geq C \frac{(N^{\frac{1}{2}}+\frac{1}{2}\kappa)^2+1-\frac{1}{4}\kappa^2}{\kappa\nu}  \\
& \geq C\frac{(\frac{1}{\sqrt{\kappa}}\aiminnorm{\Lambda^{\beta} f}e^M+\frac{1}{2}\kappa)^2}{\kappa\nu}+C\frac{1-\frac{1}{4}\kappa^2}{\kappa\nu},
\end{split}
\]
where $N, M$ are defined in \eqref{N} and \eqref{e111}, respectively. 
\end{remark}

\section*{Acknowledgments}
The authors would like to thank Prof. M. Jolly for interesting discussions on this work.
This work was partially supported by the National Science Foundation under the grant NSF DMS-1418911.3, and by the Research Fund of Indiana University.

\bibliographystyle{amsalpha}
\providecommand{\bysame}{\leavevmode\hbox to3em{\hrulefill}\thinspace}
\providecommand{\MR}{\relax\ifhmode\unskip\space\fi MR }
\providecommand{\MRhref}[2]{%
  \href{http://www.ams.org/mathscinet-getitem?mr=#1}{#2}
}
\providecommand{\href}[2]{#2}

\end{document}